\documentclass[12pt,amstex]{amsart}

\usepackage{epsfig}
\usepackage{amsmath}
\usepackage{amssymb}
\usepackage{amscd}
\usepackage{graphicx}

\newtheorem{thm}{Theorem}[section]
\newtheorem{lem}[thm]{Lemma}

\newtheorem{prop}[thm]{Proposition}
\newtheorem{ques}[thm]{Question}
\newtheorem{cor}[thm]{Corollary}

\theoremstyle{definition}
\newtheorem{de}[thm]{Definition}

\theoremstyle{remark}
\newtheorem{rem}[thm]{Remark}

\numberwithin{equation}{section}

\input xy
\xyoption{all}

\def \b {\beta}
\def \ep {\epsilon}
\def \d {\delta}
\def \w {\omega}
\def \lra {\longrightarrow}
\def \ra {\rightarrow}

\def \U {\mathcal{U}}

\def \ov {\overline}
\def \F {\mathcal{F}}

\def \M {\mathcal{M}}
\def \lra{\longrightarrow}

\def \p1 {\pi _1}
\def \p {\mathbb{P}}

\def \N {\mathbb{N}}

\def \Z {\mathbb{Z}}

\def \ra {\rightarrow}

\def \ep {\epsilon}

\def \P {\mathcal P}
\def \bfA {{\bf A}}
\def \Ind {{\rm Ind}}

\begin{document}
\title{product recurrent properties, disjointness and weak disjointness}
\author{Pandeng Dong, Song Shao and Xiangdong Ye}


\address{Department of Mathematics, University of Science and Technology of China,
Hefei, Anhui, 230026, P.R. China}
\email{dopandn@mail.ustc.edu.cn}\email{songshao@ustc.edu.cn}\email{yexd@ustc.edu.cn}

\date{Aug. 23, 2009}
\date{0912, 2009}
\date{0919, 2009}
\date{Sept. 27, 2009}
\date{Oct. 02, 2009}
\date{Oct. 06, 2009}
\date{Oct. 18, 2009}
\date{Jan. 17, 2010}

\thanks{P.D. Dong and S. Shao are supported by NNSF of
China (10871186), and X.D. Ye is supported by NNSF of China
(10531010) and 973 programm.}


\begin{abstract}
Let $\F$ be a collection of subsets of $\Z_+$ and $(X,T)$ be a
dynamical system. $x\in X$ is $\F$-recurrent if for each
neighborhood $U$ of $x$, $\{n\in\Z_+:T^n x\in U\}\in \F$. $x$ is
$\F$-product recurrent if $(x,y)$ is recurrent for any
$\F$-recurrent point $y$ in any dynamical system $(Y,S)$. It is well
known that $x$ is $\{infinite\}$-product recurrent if and only if it
is minimal and distal. In this paper it is proved that the closure
of a $\{syndetic\}$-product recurrent point (i.e. weakly product
recurrent point) has a dense minimal points; and a $\{piecewise\
syndetic\}$-product recurrent point is minimal. Results on product
recurrence when the closure of an $\F$-recurrent point has zero
entropy are obtained.

It is shown that if a transitive system is disjoint from all minimal
systems, then each transitive point is weakly product recurrent.
Moreover, it proved that each weakly mixing system with dense
minimal points is disjoint from all minimal PI systems; and each
weakly mixing system with a dense set of distal points or an
$\F_s$-independent system is disjoint from all minimal systems.
Results on weak disjointness are described when considering
disjointness.

\end{abstract}


\maketitle 

\tableofcontents


\section{Introduction}

\subsection{Dynamical preliminaries}

In the article, integers, nonnegative integers and natural numbers
are denoted by $\Z$, $\Z_+$ and $\N$ respectively. By a {\it
topological dynamical system} (t.d.s.) we mean a pair $(X,T)$, where
$X$ is a compact metric space (with metric $d$) and $T:X\to X$ is
continuous and surjective. A non-vacuous closed invariant subset $Y
\subset X$ defines naturally a {\em subsystem} $(Y,T)$ of $(X,T)$.

The {\em orbit} of $x$, $orb(x,T)$ (or simply $orb(x)$), is the set
$\{T^nx: n\in \Z_+\}=\{x,T(x),\ldots\}$. The {\em $\omega$-limit
set} of $x$, $\omega(x,T)$, is the set of all limit points of
$orb(x,T)$. It is easy to verify that $\w(x, T)=\bigcap_{n\ge 0}
\overline{\{T^i(x):i\ge n\}}$.

A t.d.s. $(X,T)$ is {\it transitive} if for each pair of opene (i.e.
nonempty and open) subsets $U$ and $V$, $N(U,V)=\{n\in\Z_+:
T^{-n}V\cap U\not=\emptyset\}$ is infinite. It is {\em point
transitive} if there exists $x\in X$ such that
$\overline{orb(x,T)}=X$; such $x$ is called a {\em transitive
point}, and the set of transitive points is denoted by $Tran_T$. It
is well known that if a compact metric system $(X,T)$ is transitive
then $Tran_T$ is a dense $G_\delta$ set. $(X,T)$ is {\it weakly
mixing} if $(X\times X, T\times T)$ is transitive.

A t.d.s $(X,T)$ is {\em minimal} if $Tran_T=X$. Equivalently,
$(X,T)$ is minimal if and only if it contains no proper subsystems.
By the argument using Zorn's Lemma  any t.d.s. $(X,T)$ contains some
minimal subsystem, which is called a {\em minimal set} of $X$. A
point $x \in X $ is {\em minimal} or {\em almost periodic} if the
subsystem $(\overline{orb(x,T)},T)$ is minimal.

\medskip

Let $(X,T)$ be a t.d.s. and $(x,y)\in X^2$. It is a {\it proximal}
pair if there is a sequence $\{n_i\}$ in $\Z_+$ such that
$\lim_{n\ra +\infty} T^{n_i} x =\lim_{n\ra +\infty} T^{ n_i} y$; and
it is a {\it distal} pair if it is not proximal. Denote by $P(X,T)$
or $P_X$ the set of all proximal pairs of $(X,T)$. A point $x$ is
said to be {\em distal} if whenever $y$ is in the orbit closure of
$x$ and $(x,y)$ is proximal, then $x = y$. A t.d.s. $(X,T)$ is
called {\it distal} if $(x,x')$ is distal whenever $x,x'\in X$ are
distinct.

A t.d.s. $(X,T)$ is {\it equicontinuous} if for every $\ep>0$ there
exists $\d>0$ such that $d(x_1,x_2)< \d$ implies
$d(T^nx_1,T^nx_2)<\ep$ for every $n\in \Z_+$. It is easy to see that
each equicontinuous system is distal.

\medskip

For a t.d.s. $(X,T)$, $x\in X$ and $U\subset X$ let
$$N(x,U)=\{n\in \Z_+: T^nx\in U\}.$$

A point $x\in X$ is said to be {\em recurrent} if for every
neighborhood $U$ of $x$, $N(x,U)$ is infinite. Equivalently, $x\in
X$ is recurrent if and only if $x\in \w(x,T)$, i.e. there is a
strictly increasing subsequence $\{n_i\}$ of $\N$ such that
$T^{n_i}x\lra x$. Denote by $R(X,T)$ the set of all recurrent points
of $(X,T)$.

\subsection{Product recurrence and weakly product recurrence}

The notion of product recurrence was introduced by Furstenberg in
\cite{F1}. Let $(X,T)$ be a t.d.s.. A point $x\in X$ is said to be
{\em product recurrent} if given any recurrent point $y$ in any
dynamical system $(Y,S)$, $(x,y)$ is recurrent in the product system
$(X\times Y, T\times S)$. By associating product recurrence with a
combinatorial property on the sets of return times (i.e. $x$ is
product recurrent if and only if it is $IP^*$ recurrent),
Furstenberg  proved that product recurrence is equivalent to
distality \cite[Theorem 9.11]{F1}. In \cite{AuF94} Auslander and
Furstenberg extended the equivalence of product recurrence and
distality to more general semigroup actions. If a semigroup $E$ acts
on the space $X$ and $F$ is a closed subsemigroup of $E$ , then
$x\in X$ is said to be {\em $F$-recurrent} if $px =x$ for some $p\in
F$, and {\em product $F$-recurrent} if whenever $y$ is an
$F$-recurrent point (in some space $Y$ on which $E$ acts) the point
$(x ,y)$ is $F$-recurrent in the product system. In \cite{AuF94} it
is shown that, under certain conditions, a point is product
$F$-recurrent if and only if it is a distal point. This subject is
also discussed in \cite{EEN}.

In \cite{AuF94}, Auslander and Furstenberg posed a question: if $(x,
y)$ is recurrent for all minimal points $y$, is $x$ necessarily a
distal point? This question is answered in the negative in
\cite{HO}. Such $x$ is called a {\em weakly product recurrent} point
there.

\medskip
The main purpose of this paper is to study a more general question,
i.e. to study a point $x$ with property that $(x,y)$ is recurrent
for any $y$ with some special recurrent property. We will also show
how this question is related to disjointness and weak disjointness.
To be more precise, we need some notions.

\subsection{Furstenberg families}

Let us recall some notions related to Furstenberg families (for
details see \cite{A, F1}). Let $\P=\P({\Z}_{+})$ be the collection
of all subsets of $\Z_+$. A subset $\F$ of $\P$ is a {\em
(Furstenberg) family}, if it is hereditary upwards, i.e. $F_1
\subset F_2$ and $F_1 \in \F$ imply $F_2 \in \F$. A family $\F$ is
{\it proper} if it is a proper subset of $\P$, i.e. neither empty
nor all of $\P$. It is easy to see that $\F$ is proper if and only
if ${\Z}_{+} \in \F$ and $\emptyset \notin \F$. Any subset
$\mathcal{A}$ of $\P$ can generate a family $[\mathcal{A}]=\{F \in
\P:F \supset A$ for some $A \in \mathcal{A}\}$. If a proper family
$\F$ is closed under intersection, then $\F$ is called a {\it
filter}. For a family $\F$, the {\it dual family} is
$$\F^*=\{F\in\P: {\Z}_{+} \setminus F\notin\F\}=\{F\in \P:F \cap F' \neq
\emptyset \ for \ all \ F' \in \F \}.$$ $\F^*$ is a family, proper
if $\F$ is. Clearly,
$$(\F^*)^*=\F\ \text{and}\ {\F}_1\subset {\F}_2 \Longrightarrow
{\F}_2^* \subset {\F}_1^*.$$ Denote by $\F_{inf}$ the family
consisting of all infinite subsets of $\Z_+$.

\subsection{$\F$-recurrence and some important families}

Let $\F$ be a family and $(X,T)$ be a t.d.s.. We say $x\in X$ is
$\F$-{\it recurrent} if for each neighborhood $U$ of $x$, $N(x,U)\in
\F$. So the usual recurrent point is just $\F_{inf}$-recurrent one.

\medskip

Recall that a t.d.s. $(X,T)$ is

\begin{enumerate}

\item[$\bullet$] an $E$-{\it system} if it is transitive and has
an invariant measure $\mu$ with full support, i.e., $supp(\mu)=X$;

\item[$\bullet$] an $M$-{\it system} if it is transitive and the
set of minimal points is dense; and

\item[$\bullet$] a $P$-system if it is transitive and the set of
periodic points is dense.
\end{enumerate}

A subset $S$ of $\Z_+$ is {\it syndetic} if it has a bounded gaps,
i.e. there is $N\in \N$ such that $\{i,i+1,\cdots,i+N\} \cap S \neq
\emptyset$ for every $i \in {\Z}_{+}$. $S$ is {\it thick} if it
contains arbitrarily long runs of positive integers, i.e. there is a
strictly increasing subsequence $\{n_i\}$ of $\Z_+$ such that
$S\supset \bigcup_{i=1}^\infty \{n_i, n_i+1, \ldots, n_i+i\}$. The
collection of all syndetic (resp. thick) subsets is denoted by
$\F_s$ (resp. $\F_t$). Note that $\F_s^*=\F_t$ and $\F_t^*=\F_s$.

Some dynamical properties can be interrupted by using the notions of
syndetic or thick subsets. For example, a classic result of
Gottschalk stated that $x$ is a minimal point if and only if
$N(x,U)\in \F_s$ for any neighborhood $U$ of $x$, and a t.d.s.
$(X,T)$ is weakly mixing if and only if $N(U,V)\in \F_t$ for any
non-empty open subsets $U,V$ of $X$ \cite{F, F1}.

A subset $S$ of $\Z_+$ is {\it piecewise syndetic} if it is an
intersection of a syndetic set with a thick set. Denote the set of
all piecewise syndetic sets by $\F_{ps}$. It is known that a t.d.s.
$(X,T)$ is an $M$-{\it system} if and only if there is a transitive
point $x$ such that $N(x,U)\in \F_{ps}$ for any neighborhood $U$ of
$x$ (see for example \cite[Lemma 2.1]{HY}).

Let $\{ b_i \}_{i\in I}$ be a finite or infinite sequence in
$\mathbb{N}$. One defines $$FS(\{ b_i \}_{i\in
I})=\Big\{\sum_{i\in \alpha} b_i: \alpha \text{ is a finite
non-empty subset of } I\Big \}.$$ $F$ is an {\it IP set} if it
contains some $FS({\{p_i\}_{i=1}^{\infty}})$, where $p_i\in\N$.
The collection of all IP sets is denoted by $\F_{ip}$. A subset of
$\N$ is called an {\it ${\text{IP}}^*$-set}, if it has non-empty
intersection with any IP-set. It is known that a point $x$ is a
recurrent point if and only if $N(x,U)\in \F_{ip}$ for any
neighborhood $U$ of $x$, and $x$ is distal if and only if $x$ is
$IP^*$-recurrent \cite{F1}.

Let $S$ be a subset of $\mathbb{Z}_+$. The {\it upper Banach
density} and {\it lower Banach density} of $S$ are
$$BD^*(S)=\limsup_{|I|\to \infty}\frac{|S\cap I|}{|I|},\ \text{and}\
BD_*=\liminf_{|I|\to \infty}\frac{|S\cap I|}{|I|},$$ where $I$
ranges over intervals of $\mathbb{Z}_+$, while the {\it upper
density} of $S$ is
$$D^*(S)=\limsup_{n\to \infty}\frac{|S\cap [0,n-1]|}{n}.$$
Let $\F_{pubd}=\{S\subseteq \Z_+: BD^*(S)>0\}$ and
$\F_{pd}=\{S\subseteq \Z_+: D^*(S)>0\}$. It is known a t.d.s.
$(X,T)$ is an $E$-{\it system} if and only if there is a transitive
point $x$ such that $N(x,U)\in \F_{pubd}$ for any neighborhood $U$
of $x$ (see for example \cite[Lemma 3.6]{HKY}).

\subsection{$\F$-product recurrence and disjointness}

Let $\F$ be a family. For a t.d.s. $(X,T)$, $x\in X$ is $\F$-{\em
product recurrent} if given any $\F$-recurrent point $y$ in any
t.d.s $(Y,S)$, $(x,y)$ is recurrent in the product system $(X\times
Y, T\times S)$. Note that $\F_{inf}$-product recurrence is nothing
but product recurrence; and $\F_{s}$-product recurrence is weak
product recurrence. In this paper we will study the properties of
$\F$-product recurrent points, especially when $\F=\F_{pubd}$,
$\F_{ps}$, or $\F_s$.

\medskip

The notion of {\it disjointness} of two t.d.s. was introduced by
Furstenberg his seminal paper \cite{F}. Let $(X,T)$ and $(Y,S)$ be
two t.d.s.. We say $J\subset X\times Y$ is a {\it joining} of $X$
and $Y$ if $J$ is a non-empty closed invariant set, and is projected
onto $X$ and $Y$ respectively. If each joining is equal to $X\times
Y$ then we say that $(X,T)$ and $(Y,S)$ are {\it disjoint}, denoted
by $(X,T)\perp (Y,S)$ or $X\perp Y$. Note that if $(X,T)\perp (Y,S)$
then one of them is minimal \cite{F}, and if $(X,T)$ is minimal then
the set of recurrent points of $(Y,S)$ is dense \cite{HY}.

In \cite{F}, Furstenberg showed that each totally transitive system
with dense set of periodic points is disjoint from any minimal
system; each weakly mixing system is disjoint from any minimal
distal system. He left the following question:

\medskip
\noindent {\bf Problem}: {\em Describe the system who is disjoint
from all minimal systems.}

\subsection{Main results of the paper}
It turns out that if a transitive t.d.s. $(X,T)$ is disjoint from
all minimal t.d.s. then each transitive point of $(X,T)$ is a weak
product recurrent one (Theorem \ref{disjoint}). Thus, by \cite{HY}
it is not necessarily minimal. Moreover, it is proved that the orbit
closure of each weak product recurrent point is an $M$-system, i.e.
with a dense set of minimal points (Theorem \ref{orbitM}). Contrary
to the above situation it is shown that an $\F_{ps}$-product
recurrent point is minimal (Theorem \ref{thickM}).

Results on product recurrence when the closure of an $\F$-recurrent
point has zero entropy are obtained. It is shown that if $(x,y)$ is
recurrent for any point $y$ whose orbit closure is a minimal system
having zero entropy, then $x$ is $\F_{pubd}$-recurrent (Theorem
\ref{orbitE}); and if $(x,y)$ is recurrent for any point $y$ whose
orbit closure is an $M$-system having zero entropy, then $x$ is
minimal (Theorem \ref{thick0}). Moreover, it turns out that if
$(x,y)$ is recurrent for any recurrent $y$ whose orbit closure has
zero entropy, then $x$ is distal (Theorem \ref{thm5.2}).

Several results on disjointness are obtained, and results on weak
disjointness are described when considering disjointness. For
example, it is proved that a weakly mixing system with dense minimal
points is disjoint from all minimal PI systems (Theorem \ref{pid});
and a weakly mixing system with a dense set of distal points or an
$\F_s$-independent t.d.s. is disjoint from any minimal t.d.s.
(Theorem \ref{huangidea} and \ref{independence}). Moreover, it is
shown that if a transitive t.d.s. is disjoint from all minimal
weakly mixing t.d.s. then it is an $M$-system (Proposition
\ref{propwm}).

\subsection{Organization of the paper}

The paper is organized as follows: In Section 2 we discuss
recurrence and product recurrence. We begin with Hindman Theorem and
rebuilt Furstenberg's result about product recurrence. In Section 3
we study $\F_{ps}$-product recurrence and show any $\F_{ps}$-product
recurrent point is minimal. In Section 4 we aim to show that the
closure of an $\F_s$-product recurrent point is an $M$-system. On
the way to do this, we show that if $(X,T)$ is a transitive t.d.s.
which is disjoint from any minimal system, then each point in
$Tran_T$ is $\F_s$-product recurrent.
In Section 5 we study $\F$-product recurrence with zero entropy. We
discuss properties concerning extensions and factors in Section 6.
We study disjointness and weak disjointness in Section 7. In Section
8 we discuss some more generalizations of the notions concerning
product recurrence. Finally in the Appendix we discuss relative
proximal cells.

\medskip
\noindent{\bf Acknowledgement:} We thank E. Glasner, W. Huang,  H.
Li, and W. Ott for useful discussion over the topic. Particularly,
we thank Huang for allowing us including a proof of a disjoint
result (Theorem \ref{huangidea}) and for useful comments on various
versions of the paper. After finishing the paper we received a
preprint by P. Oprocha who also proved Theorem \ref{huangidea}.

\section{Recurrence and product recurrence}
It is known that $x$ is distal if and only if $(x,y)$ is recurrent
for any recurrent point $y$ \cite{F1}. The usual proof uses the
Auslander-Ellis theorem which states that if $(X,T)$ is a t.d.s. and
$x\in X$ then there is a minimal point $y\in \overline{orb(x,T)}$
such that $(x,y)$ is proximal. Usually one proves the
Auslander-Ellis theorem by using the Ellis semigroup theory. In this
section we give a proof of the theorem without using the Ellis
semigroup theory.

\subsection{Recurrence and IP-set} In this subsection Hindman Theorem
is used to prove Auslander-Ellis Theorem. Also some interesting
relations between recurrence and IP-set will be built.

\begin{thm}[Hindman, \cite{Hi74}]
For any finite partition of an IP-set, one of the cells of the
partition is an IP-set.
\end{thm}

The following lemma is basically due to Furstenberg, see
\cite{F1}.

\begin{lem}\label{lem2.2}
Let $(X,T)$ be a compact metric t.d.s.. If $x \in R(X,T)$ and
$\{V_i\}_{i=1}^{\infty}$ is a collection of neighborhoods of $x$,
then there is some IP set $FS(\{p_i\}_{i=1}^{\infty})$ such that
$FS(\{p_i\}_{i=n}^{\infty}) \subset N(x,V_n) $ for all $n\in \N$.
Especially, each recurrent point is $\F_{ip}$-recurrent.
\end{lem}

\begin{proof}
We prove the lemma using induction. Since $V_1$ is a neighborhood of
$x$ and $x$ is recurrent, there is some $p_1\in \N$ such that
$$T^{p_1}x\in V_1.$$ As $V_1, T^{-p_1}V_1, V_2$ are neighborhoods of
$x$, so is their intersection $V_1\cap T^{-p_1}V_1\cap V_2$. And by
the recurrence of $x$ there is some $p_2\in \N$ such that
$$T^{p_2}x\in V_1\cap T^{-p_1}V_1\cap V_2.$$ Hence $$T^{p_1}x,
T^{p_2}x, T^{p_1+p_2}x\in V_1,$$ and
$$T^{p_2}x \in V_2.$$

Now for $n\in \N$ assume that we have a finite sequence
$p_1,p_2,\ldots, p_n$ such that
\begin{equation}\label{}
    FS(\{p_i\}_{i=j}^n)\subseteq N(x,V_j), j=1,2,\ldots, n.
\end{equation}
That is, for each $j=1,2,\ldots, n$
\begin{equation*}
    T^mx\in V_j,\quad \forall m\in FS(\{p_i\}_{i=j}^n).
\end{equation*}
Hence $\left( \bigcap_{j=1}^n \bigcap_{m\in FS(\{p_i\}_{i=j}^n)}
T^{-m}V_j \right)\cap \bigcap_{i=1}^{n+1}V_i$ is a neighborhood of
$x$. Take $p_{n+1}\in \N$ such that
\begin{equation*}
    T^{p_{n+1}}x \in \left( \bigcap_{j=1}^n \bigcap_{m\in FS(\{p_i\}_{i=j}^n)}
T^{-m}V_j \right)\cap \bigcap_{i=1}^{n+1}V_i.
\end{equation*}
Then for each $j=1,2,\ldots, n+1$
\begin{equation*}
    T^mx\in V_j,\quad \forall m\in FS(\{p_i\}_{i=j}^{n+1}).
\end{equation*}
That is
\begin{equation*}
    FS(\{p_i\}_{i=j}^{n+1})\subseteq N(x,V_j), j=1,2,\ldots, n+1.
\end{equation*}
So inductively we have an IP set $FS(\{p_i\}_{i=1}^{\infty})$ such
that $FS(\{p_i\}_{i=n}^{\infty}) \subset N(x,V_n) $ for all $n\in
\N$. And the proof is completed.
\end{proof}

Let $(X,T)$ be a t.d.s. and $A\subseteq \Z_+$ be a sequence. Write
$$T^Ax=\{T^nx: n\in A\}$$ and let $A-n=\{m-n:m\in A,m-n\geq 1\}$ for
$n\in\Z_+$.

Using the method from \cite{Blokh}, we have

\begin{lem}\label{lem2.3}
Let $(X,T)$ be a compact metric t.d.s. and
$Q=FS(\{p_i\}_{i=1}^{\infty})$. For any $x \in X$ there is some $y
\in \overline {T^Q x} \cap R(X, T)$ and $\{p_{n_i}\}_{i=1}^{\infty}
\subseteq \{p_i\}_{i=1}^{\infty}$ such that for any neighborhood $U$
of $y$ there is some $j$ with $FS(\{p_{n_i}\}_{i=j}^{\infty})
\subseteq N(y,U)$ and $(x,y) \in P(X,T)$.
\end{lem}

\begin{proof}
Set $K_1=\overline {T^Px}, P_1=Q$ and $p_{n_i} \in
\{p_i\}_{i=1}^{\infty}$. Then $$P_1\cap (P_1-p_{n_1}) \supseteq
FS(\{p_i\}_{i\neq n_1}).$$ Hence
$$K_1 \cap T^{-p_{n_1}}K_1 \supseteq \ov {T^{P_1\cap (P_1-p_{n_1})}x}.$$
Let $K_1 \cap T^{-p_{n_1}}K_1=\displaystyle \bigcup
_{i=1}^{r_1}K_{1,i},$ where $K_{1,i}$ is compact and $diam K_{1,i}
<\frac 12$. So we have
$$P_1\cap (P_1-p_{n_1})=\displaystyle \bigcup_{i=1}^{r_1}\{n \in P_1\cap
(P_1-p_{n_1}) : T^nx \in K_{1,i}\}.$$ By Hindman Theorem there is
some $j$ such that $$P_2=\{ n \in P_1\cap (P_1-p_{n_1}) : T^nx \in
K_{1,j}\}$$ is an IP subset of $P_1\cap (P_1-p_{n_1})$. And we set
$K_2=K_{1,j}$. Clearly, $K_2\subseteq K_1$, $diam K_2<\frac 12$,
$T^{p_{n_1}}K_2\subseteq K_1$ and $T^{P_2}x \subseteq K_2$.

Continuing inductively, we have $\{p_{n_i}\}\subseteq \{p_i\}$, IP
sets $P_1 \supseteq P_2\supseteq \cdots$ and compact sets $K_1
\supseteq K_2\supseteq \cdots$ such that $diam K_j <\frac 1j$,
$p_{n_j}\in P_j$, $T^{p_{n_j}}K_{j+1} \subseteq  K_j$ and $T^{P_j}x
\subseteq K_j$. Let $y \in \bigcap_{i=1}^{\infty} K_i $. It is easy
to check that $y$ is the point we look for.
\end{proof}

\begin{prop}\label{prop2.4}
Let $(X,T)$ be a compact metric t.d.s.. If $(Y,S)$ is another t.d.s.
and  $z \in R(Y,S)$, then for any $x\in X$ there is some $y \in
\overline {orb(x, T)}$ such that $(x,y)\in P(X,T)$ and $(y,z)$ is a
recurrent point of $X\times Y$.
\end{prop}

\begin{proof}
Let $\{V_n\}_{n=1}^{\infty}$ be neighborhood basis of $z$. By Lemma
\ref{lem2.2} there is some IP set $Q=FS(\{p_i\}_{i=1}^{\infty})$
such that $FS(\{p_i\}_{i=n}^{\infty}) \subset N(z,V_n) $ for all
$n\in \N$. Let $y$ be the recurrent point described in Lemma
\ref{lem2.3}. Then for any neighborhoods $U,V$ of $y, z$ we have
$$N((y, z),U\times V)=N(y,U)\cap N(z,V)\neq \emptyset.$$
Hence $(y, z)$ is a recurrent point of $X\times Y$.
\end{proof}

\begin{thm}[Auslander-Ellis]\label{thmAuslander}
Let $(X,T)$ be a compact metric t.d.s.. Then for any $x\in X$ there
is some minimal point $y \in \overline {orb(x, T)}$ such that
$(x,y)$ is proximal.
\end{thm}

\begin{proof}
Without loss of generality, we assume $x$ is not minimal. Then there
is some minimal set $Y$ in $\overline {orb(x)}$. Now we will find a
thick $A$ such that $\overline {T^Ax} \setminus T^Ax \subseteq Y$.
Then taking any IP subset $Q$ from $A$, by Lemma \ref{lem2.3} there
is some $y \in \overline {T^Q x} \cap R(X,T)$ and $(x,y) \in
P(X,T)$. Since $y \in \overline {T^Qx} \setminus T^Qx \subseteq Y$,
$y$ is a minimal point. Thus we finish our proof.

It remains to find a thick $A$ such that $\overline {T^Ax} \setminus
T^Ax \subseteq Y$. Let $V_n=\{z \in X: d(z,Y)<\frac 1n\}$ and then
$\{V_n\}_{n=1}^{\infty}$ is a neighborhood basis of $Y$. Let $\delta
_{n} >0$ such that $d(T^ix',T^ix'')<\frac 1n, i=0,1,\cdots,n-1$ if
$d(x',x'')<\delta_n$. As $Y \subseteq \overline {orb(x, T)}$ there
is some $i_n$ such that $d(T^{i_n}x,Y)<\delta_n$. Then by the
invariance of $Y$, $d(T^{i_n+j}x,Y)<\frac 1n, j=0,1,\cdots,n-1$. Set
$A=\displaystyle \bigcup _{n=1}^{\infty}\{i_n+j\}_{j=0}^{n-1}$. By
our construction we have $\overline {T^Ax} \setminus T^Ax \subseteq
Y$.
\end{proof}

\begin{rem}\label{rem2.6}

\noindent (1) The previous proofs of Theorem \ref{thmAuslander}
involve the use of Zorn's Lemma. Here for a compact metric space we
get a proof only using Hindman Theorem. Note that usually to show
that any t.d.s. $(X,T)$ contains some minimal subsystem is to use
the well-known Zorn's Lemma argument. But for the case when $X$ is
metric and the action semigroup is $\Z_+$ Weiss \cite{Weiss} gave a
constructive proof.

\medskip

\noindent (2) From Auslander-Ellis Theorem Furstenberg introduced a
notion called central set.  A subset $S\subseteq \Z_+$ is a {\em
central set} if there exists a system $(X,T)$, a point $x\in X$ and
a minimal point $y$ proximal to $x$, and a neighborhood $U_y$ of $y$
such that $N(x,U_y)\subset S$. It is known that any central set is
an IP-set \cite[Proposition 8.10.]{F1}.

\medskip

\noindent (3) By Lemma \ref{lem2.2} $x$ is a recurrent point if and
only if it is $\F_{ip}$-recurrent. In \cite[Theorem 2.17]{F1} it is
also shown that for any IP-set $R$ there exists a t.d.s. $(X,T)$, a
recurrent $x\in X$ and a neighborhood $U$ of $x$ such that
$N(x,U)\subseteq R\cup \{0\}$.
\end{rem}

\subsection{Product recurrence}

 The following proposition was proved in \cite[Theorem 9.11.]{F1}
and we give a proof for completeness.

\begin{prop}\label{distal}
Let $(X,T)$ be a t.d.s.. The following statements are equivalent:

\begin{enumerate}

\item $x$ is distal.

\item $x$ is product recurrent.

\item $(x,y)$ is minimal for each minimal point $y$ of a system $(Y,S)$.

\item $x$ is $IP^*$-recurrent.

\end{enumerate}

\end{prop}

\begin{proof}Denote $X=\overline{orb(x, T)}$.
First by Remark \ref{rem2.6} it is easy to see that
(2)$\Longleftrightarrow$(4).

(1) $\Longrightarrow$ (4). If $x$ is not $IP^*$-recurrent, then
there is a neighborhood $U$ of $x$ such that $N(x,U)$ is not an
$IP^*$-set, i.e. there exists an IP-set $Q$ such that $T^Q x\cap
U=\emptyset$. By Lemma \ref{lem2.3}, we know that there is a point
$y\in\overline{T^Q x}$ i.e. $y\not\in U$ such that $(x,y) \in
P(X,T)$ which contradicts the assumption that $x$ is distal.

(4) $\Longrightarrow$ (1). As any thick set contains an IP-set, we
get that $x$ is a minimal point. If $x$ is not distal, there exists
a different point $x'\in X$ such that $(x,x') \in P(X,T)$. Let $U$
and $U'$ be any neighborhoods of $x$ and $x'$ which are disjoint.
$N(x,U')$ is a central set and contains an IP-set, so $N(x,U)\cap
N(x,U')\neq \emptyset$ which implies $x=x'$.

(1) $\Longrightarrow$ (3). Let $y$ be a minimal point of $(Y,S)$. If
$(x,y)$ is not minimal, by Theorem \ref{thmAuslander} there exists a
minimal point $(x',y')\in \overline{orb((x,y), T\times S)}$ which is
proximal to $(x,y)$. It follows that $x'$ is proximal to $x$ which
implies $x=x'$. For any neighborhood $U\times V$ of $(x,y)$,
$N(x,U)$ is an $IP^*$-set and $N(y',V)$ is a central set as $y'$ is
proximal to minimal point $y$, so we know that $N(x,U)\cap
N(y',V)\neq \emptyset$, i.e. $(x,y)\in\overline{orb((x,y'), T\times
S)}$ which implies that $(x,y)$ is a minimal point.

(3) $\Longrightarrow$ (1). It is easy to see that $x$ is a minimal
point. If there exists a point $x'\in \overline{orb(x, T)}$ which is
proximal to $x$, then there exists a point $(y,y)\in
\overline{orb((x,x'), T\times T)}$. As $(x,x')$ is a minimal point,
then $(x,x')\in \overline{orb((y,y), T\times T)}$ which implies
$x=x'$, so $x$ is distal.
\end{proof}

\section{$\F_{ps}$-product recurrent points}
In this section we aim to show that if $x$ is an $\F_{ps}$-product
recurrent point then it is minimal.

\begin{de}
Let $(X,T)$ be a t.d.s. and $\F$ be a family. $x\in X$ is {\em
$\F$-product recurrent} ($\F$-PR for short) if given any
$\F$-recurrent point $y$ in any t.d.s. $(Y,S)$, $(x,y)$ is recurrent
in the product system $(X\times Y, T\times S)$.
\end{de}

By definition we have the following observation immediately.

\begin{lem}
Let $\F_1, \F_2$ be two families with $\F_1\subseteq \F_2$. Then
each $\F_2$-PR point is $\F_1$-PR.
\end{lem}

It is clear that
$$\F_{inf}-PR\ \Rightarrow \  \F_{pubd}-PR \ \Rightarrow\
\F_{ps}-PR\ \Rightarrow\ \F_s-PR.$$

It was shown in \cite{HO} that an $\F_s$-PR point is not necessarily
minimal (more examples will be given in the next section). A natural
question is: if $x$ is $\F_{ps}$-PR, is $x$ minimal? Before
continuing discussion, we need some preparation about symbolic
dynamics. Let $\Sigma_2=\{0,1\}^{\Z_+}$ and $\sigma:\Sigma_2
\longrightarrow \Sigma_2$ be the shift map, i.e. the map
$$(x_0,x_1,x_2,x_3,\ldots)\mapsto(x_1,x_2,x_3,\ldots)\in \Sigma_2.$$
A {\em shift space} $(X,\sigma)$ is a subsystem of
$(\Sigma_2,\sigma)$. For any $S\subset \Z_+$, we denote by $1_S\in
\{0,1\}^{\Z_+}$ the indicator function of $S$, i.e. $1_S(a)=1$ if
$a\in S$ and $1_S(a)=0$ if $a\not\in S$. For finite blocks
$A=(a_1,\ldots, a_n)\in \{0,1\}^n$ and $B=(b_1, \ldots, b_n)\in
\{0,1\}^n$ we say $A\le B$ if $a_i\le b_i$ for each
$i\in\{1,2,\cdots,n\}$. For finite blocks $A$ and $B$ we denote the
length of $A$ by $|A|$, $\underbrace{AA\cdots A}\limits_n$ by $A^n$
for $n\in\N$ (in particular $0^n=\underbrace{00\cdots 0}\limits_n$),
and the concatenation of $A$ and $B$ by $AB$. If $(X,\sigma)$ is a
shift space, let $[i]=[i]_X=\{x\in X:x(0)=i\}$ for $i=0,1$, and
$[A]=[A]_X=\{x\in X:x_0 x_1\cdots x_{(|A|-1)}=A\}$ for any finite
block $A$.

To settle down the question we need the following notion. By an {\it
md-set} $A$ we mean there is an $M$-system $(Y,S)$, a transitive
point $y\in Y$ and a neighborhood $U$ of $y$ such that $A=N(y,U)$.

\begin{prop}\label{thick}
Every thick set containing $0$ contains an md-set.
\end{prop}
\begin{proof} Let $C\subset \Z_+$ be a thick set with $0\in C$.
Let $x=1_C=(x_0,x_1, \ldots)\in\{0,1\}^{\Z_+}$.

\medskip

By the assumption $x_0=1$ and there are $p_n< q_n\in \N$ with
$\underbrace{11\cdots 1}\limits_n\le (x_{p_n}, \ldots, x_{q_n})$ for
any $n\in\N$. It is clear that there is $a_1\ge 1$ such that
$$A_1=10^{a_1}1\le (x_0\ldots x_{l_1})$$ with $l_1=|A_1|-1$. By the same reasoning
there is $a_2>a_1$ and $a_2$ can be divided by $|A_1|$ with
$$A_2=A_10^{a_2}A_1\le (x_0,\ldots, x_{l_2})$$ where
$l_2=|A_2|-1$. Then $|A_2|$ can be divided by $|A_1|$.


Inductively assume that $A_1, \ldots, A_k$ are defined, then there
is $a_{k+1}>a_{k}$ and $a_{k+1}$ can be divided by $|A_{k}|$ with
$$A_{k+1}=A_k0^{a_k}A_kA_{k-1}^{n_{k+1}^{k-1}}\ldots
A_2^{n^2_{k+1}}A_1^{n^1_{k+1}}\le (x_0, \ldots, x_{l_{k+1}})$$ where
$|A_{1}|^{n^1_{k+1}}=|A_2|^{n^2_{k+1}}=\ldots=|A_{k-1}|^{n_{k+1}^{k-1}}=|A_k|$
and $l_{k+1}=|A_{k+1}|-1$. Then $|A_{k+1}|$ can be divided by
$|A_{j}|$ for $1\le j\le k$. It is easy to see that $\forall
i\in\mathbb{N}, n^{i}_{j}\rightarrow\infty$ when
$j\rightarrow\infty$.

Let $y=\lim_{k\to\infty} A_k \in \{0,1\}^{\Z_+}$, then $y$ is a
recurrent point under the shift $\sigma$. It is clear that
$N(y,[A_n])$ is piecewise syndetic. Thus the orbit closure of $y$ is
an $M$-system (in fact it is a $P$-system). At the same time,
$$N(y, [1])=\{n\in\Z_+: \sigma^n y\in [1]\}\subset C.$$
This completes the proof.
\end{proof}

Now we give a positive answer to the question.

\begin{thm}\label{thickM}
Let $(X,T)$ be a t.d.s.. If $x$ is $\F_{ps}$-PR, then it is minimal.
\end{thm}

\begin{proof} If $x$ is not minimal, then there is a neighborhood
$U$ of $x$ such that $N(x,U)$ is not syndetic. Thus, $\Z_+\setminus
N(x,U)$ is thick. Let $C=\{0\}\cup \Z_+\setminus N(x,U)$. By
Proposition \ref{thick}, $C$ contains a subset $A=N(y, V)$, where
$y$ is a transitive point of some $M$-system,which is
$\F_{ps}$-recurrent, and $V$ is a neighborhood of $y$. Then
$$N((x,y), U\times V)=N(x,U)\cap N(y, V)\subset \{0\},$$
which implies that $(x,y)$ is not recurrent, a contradiction. Thus
$x$ is minimal.
\end{proof}

Since each $\F_{pubd}$-PR point is an $\F_{ps}$-PR one, as a
corollary of Theorem \ref{thickM}, each $\F_{pubd}$-PR point is
minimal. Generally, we have

\begin{cor}
Let $\F$ be a family with $\F_{ps}\subseteq \F$. Then each $\F$-PR
point is minimal.
\end{cor}


\section{$\F_s$-product recurrent points}
In this section we aim to show that the closure of an $\F_s$-product
recurrent point is an $M$-system. On the way to do this, we show
that if $(X,T)$ is a transitive t.d.s. which is disjoint from any
minimal system, then each transitive point of $(X,T)$ is $\F_s$-PR.
Thus combining results from \cite{HY} we reprove that an $\F_s$-PR
point is not necessarily minimal which was obtained in \cite{HO}.
Note that weak product recurrence is also discussed in \cite{PO}
recently.

\subsection{$\F_s$-product recurrence}

\begin{de} A subset $A$ of $\Z_{+}$ is called an {\it m-set},
if there exist a minimal system $(Y,S)$, $y\in Y$ and a non-empty
open subset $V$ of $Y$ such that $A\supset N(y,V)$. The family
consisting of all m-sets is denoted by $\F_{mset}$.

A subset $A$ of $\Z_{+}$ is called an {\it sm-set} (standing for
standard m-set), if there exist a minimal system $(Y,S)$, $y\in Y$
and an open neighborhood $V$ of $y$ such that $A\supset N(y,V)$. The
family consisting of all sm-set is denoted by $\F_{smset}$.
\end{de}

It is clear that $\F_{smset}\subset \F_{mset}$ and hence
$\F_{mset}^*\subset \F_{smset}^*$. We will show (Proposition
\ref{tsyn}) that $\F_{smset}^*\subset \F_{ps}$. Moreover, we have
the following observation.

\begin{prop}\label{dis-w}
The following statements hold.
\begin{enumerate}
\item Let $(X, T)$ be transitive and $x\in Tran_T$. Then
$(X,T)$ is disjoint from any minimal t.d.s. if and only if
$N(x,U)\cap A\not=\emptyset$ for each neighborhood $U$ of $x$ and
each m-set $A$, i.e. $N(x,U)\in \F_{mset}^*$.

\item A point $x$ is $\F_s$-PR if and only if for each open neighborhood $U$ of
$x$ and each sm-set $A$, $N(x,U)\cap A\not=\emptyset$, i.e.
$N(x,U)\in \F_{smset}^*$.
\end{enumerate}
\end{prop}

\begin{proof}
(1) is proved in \cite{HY}. (2) follows from the
definitions.
\end{proof}

So we have

\begin{thm}\label{disjoint}
Let $(X,T)$ be a transitive t.d.s. which is disjoint
from any minimal system. Then each point in $Tran_T$ is $\F_s$-PR
and non-minimal. 
\end{thm}

\begin{proof}
It follows by Proposition \ref{dis-w} directly. We give a direct
argument here. Let $x\in Tran_T$ and $(Y,S)$ be a given minimal
t.d.s.. For $y\in Y$ let $A=\overline{ orb \big((x,y), T\times S
\big)}$. It is clear that $A$ is a joining and hence $A=X\times Y$.
This implies that $(x,y)$ a recurrent point of $(X\times Y, T\times
S)$ and hence $x$ is $\F_s$-PR.
\end{proof}


\medskip

For a t.d.s. $(X,T)$, $x\in X$ is a {\it regular minimal point} if
for each neighborhood $U$ of $x$, there is $k=k(U)$ such that
$N(x,U)\supset k\Z_+$. In \cite{HY} Huang and Ye showed that any
weakly mixing t.d.s. with a dense regular minimal points is disjoint
from any minimal t.d.s.. There are a lot of non-minimal systems with
this properties, for example the full shift and the example
constructed in \cite{HY}. Thus an $\F_s$-PR point is not necessarily
minimal. We note that this result was also obtained in \cite{HO}. So
naturally one would ask: if $x$ is $\F_s$-PR and not minimal, what
can we say about the properties of such point? In fact we will show
that the closure of $x$ is an $M$-system, i.e. it has a dense
minimal points.

The way we answer the question is that we will show every thickly
syndetic set containing $\{0\}$ contains an m-set. Note that a
subset $A$ of $\Z_+$ is {\it thickly syndetic} if it has non-empty
intersection with any piecewise syndetic set. More precisely, a
subset of $\Z_+$ is thickly syndetic if for each $n\in\N$ there is a
syndetic subset $S_n=\{s^n_1,s^n_2, \ldots\}$ such that $S\supset
\bigcup_{n=1}^\infty\bigcup_{i=1}^\infty\{s^n_i+1, s^n_i+2, \ldots,
s^n_i+n\}.$

For a transitive system whether it is disjoint from all minimal
systems can be checked through m-sets, for the details see
\cite{HY}. Particularly the authors showed that every thickly
syndetic set contains an m-set. To solve our question we need to
show

\begin{prop}\label{tsyn}
Every thickly syndetic set containing $\{0\}$ contains an sm-set.
\end{prop}

Since the proof of Proposition \ref{tsyn} is a little long, we left
it to the next subsection. Now we have

\begin{thm}\label{orbitM}
The orbit closure of an $\F_s$-PR point is an $M$-system.
\end{thm}

\begin{proof} Let $x$ be an $\F_s$-PR point and $U$ be an open
neighborhood of $x$. If $N(x,U)$ is not piecewise syndetic, then
$A=\Z_+\setminus N(x,U)$ is thickly syndetic. Then by Proposition~
\ref{tsyn}, $A\cup\{0\}$ contains $N(y, V)$, where $(Y, S)$ is a
minimal set, $y\in Y$ and $V$ is an open neighborhood of $y$. Thus
we have $$N((x,y), U\times V)=N(x, U)\cap N(y, V)\subset \{0\},$$
which implies that $(x,y)$ is not recurrent, a contradiction.
\end{proof}

\begin{rem}
Recall that two t.d.s. $(X,T)$ and $(Y,S)$ are {\it weakly disjoint}
if $(X\times Y, T\times S)$ is transitive. A t.d.s. is {\it
scattering} if it is weakly disjoint from all minimal t.d.s.
\cite{BHM}. We remark that a transitive point in a non-minimal
scattering t.d.s. is not necessarily weakly product recurrent, since
there is an almost equicontinuous scattering t.d.s. which is not an
$M$-system, see \cite[Theorem 4.6]{HY4}. It is worth to note that
when considering weak disjointness the return time sets $N(U,V)$
play the crucial role, but this is not the case when considering
disjointness or weak product recurrence, where sets $N(x,U)$ play
the role.
\end{rem}

We also have the following remark.
\begin{rem} \label{rem4.7}
It is easy to see that if $x$ is weakly product recurrent and $y$ is
distal, then $(x,y)$ is also weakly product recurrent. This implies
that $\ov{orb(x,y)}$ is not necessarily weakly mixing. Thus, the
collection of sm-sets is strictly contained in the collection of
m-sets, since if $(X,T)$ is transitive and is disjoint from all
minimal t.d.s. then $(X,T)$ is weakly mixing, see \cite{HY}.
\end{rem}

\subsection{Proof of Proposition \ref{tsyn}}

Let $F\subset \Z_+$ be a thickly syndetic subset containing $\{0\}$.
We will construct $y^n=1_{F_n}\in \{0,1\}^{\Z_+}$ such that
$F_n\subset F$ and $y=\lim_{n\to \infty} y^n=1_A$ is a minimal
point. Then let $Y=\overline{orb(y,\sigma)}$ and $[1]=\{x\in Y:
x(0)=1\}$. Since $A\subset F$ and $A=N(y,[1])$, we have the theorem.

To obtain $y^n$ we construct a finite word $A_n$ such that $y^n$
begins with $A_n$, $A_n$ appears in $y^n$ syndetically and $A_{n+1}$
begins with $A_n$. The reason we can do this is that $1^n=(1,\ldots,
1)\ (n\ \text{times})$ appears in $1_F$ syndetically for each
$n\in\N$. More precisely we do as follows.

\medskip
\noindent{\bf Step 1:} Construct $A_1$ and $F_1\subset F$ such
that $A_1$ appears in $y^1=1_{F_1}$ with gaps bounded by $l_1$ and
$y^1$ begins with $A_1$.

\medskip
Let $\min F=a_1-1$ and $A_1=1_F[0;a_1-1]$. Set $B_1=A_1A_10A_1$ and
$r_1=b_1=|B_1|=3a_1+1$. As $F$ is thickly syndetic, $1^{r_1}$
appears in $F$ at a syndetic set $W_1=\{w^1_1,w^1_2,\ldots\}$.
Without loss of generality assume that $2r_1\le w_{j+1}^1-w_j^1\le
l_1$ and $2k_1\le w_1^1\le l_1$, where $l_1$ is some number in $\N$.
Put $u_i^1=w_i^1,i\in \N$. Choose $y^1\in\{0,1\}^{\Z_+}$ such that

\begin{enumerate}
\item[$\bullet$] $y^1[0;a_1-1]=A_1$, $y^1[u_i^1;u_i^1+b_1-1]=B_1$ and

\item[$\bullet$] $y^1(j)=0$ if $j\in \Z_+\setminus
([0;a_1-1]\cup\cup_{i=1}^\infty [u_i^1;u_i^1+b_1-1])$.
\end{enumerate}

It is easy to see that $B_1$ as well as $A_1$ appears in $y^1$ with gaps bounded by
$l_1$ and $F_1\subset F$, where $1_{F_1}=y^1$.

\medskip

\noindent{\bf Step 2:} Construct $A_2$ and $F_2\subset F$ such
that
\begin{enumerate}
\item $A_2$ has the form of $A_1V_1B_1$ and if $a_2=|A_2|$ then
$A_2=y^1[0;a_2-1]$.

\item $y^2[0;a_2-1]=A_2$ and $A_1, A_2$ appear in $y^2$
syndetically with gaps bounded by $l_1$ and $l_2$ respectively.

\item $F_2=\{i\in\Z_+:y^2(i)=1\}\subset F$.
\end{enumerate}

Set $a_2=u^1_1+b_1$ and let $A_2=y^1[0;a_2-1]$, $B_2=A_2A_20A_2$,
$b_2=|B_2|=3a_2+1$. Then $A_2$ has the form of $A_1V_1B_1$. Let
$r_2=2l_1+2b_1+b_2$. As $F$ is thickly syndetic, $1^{r_2}$ appears
in $F$ at a syndetic set $W_2=\{w^2_1,w^2_2,\ldots\}$. Without loss
of generality assume that $2r_2\le w_{j+1}^2-w_j^2\le l_2-
(l_1+b_1)$ and $2a_2\le w_1^2\le l_2-(l_1+b_1)$, where $l_2$ is some
number in $\N$.

To get $y^2$ we change $y^1$ at places $[w^2_i;w^2_i+r_2-1]$ for
each $i\in\N$. It is enough to show the idea how we do at
$[w^2_1;w^2_1+r_2-1]$.

Let $k,j$ satisfy that $u_{k-1}^1<w^2_1\le u_k^1$ and
$u_j^1+b_1-1\le w^2_1+r_2-1<u^1_{j+1}+b_1-1$. Let $l$ be the integer
part of $(u^1_j-1-u^1_k-b_1-b_2)/b_1$.

Put $u_1^2=u_k^1+b_1$. Let $y^2[u_1^2; u_1^2+b_2-1]=B_2$ and
$y^2[u_1^2+b_2+pb_1; u_1^2+b_2+(p+1)b_1-1]=B_1$ for
$p=0,1,\ldots,l-1$. That is, first we put $B_2$ at place $u_1^2$ and
then we put as many as $B_1$ we can. We do the same at all places
$[w^2_i;w^2_i+r_2-1]$, we get $u_i^2\in [w_i^2,w_i^2+r_2-1]$ with
$y^2[u_i^2;u_i^2+b_2-1]=A_2$, $i=1,2,\ldots$.

In such a way we get $y^2$. It is easy to see that $y^1$ and $y^2$
differ possibly at $[w^2_i;w^2_i+r_2-1]$. Thus
$$F_2=\{i\in\Z_+:y^2(i)=1\}\subset F_1\bigcup\cup_{i=1}^\infty
[w^2_i;w^2_i+r_2-1].$$ At the same time $B_1, B_2$ appear in $y^2$
syndetically with gaps bounded by $l_1$ and $l_2$ respectively by
the construction and so are $A_1$, $A_2$.

\medskip
\noindent{\bf Step 3:} Construct $A_{m+1}$ and $F_{m+1}\subset F$
inductively such that
\begin{enumerate}
\item $A_{m+1}$ has the form of $A_mV_mB_m$ and if $a_{m+1}=
|A_{m+1}|$ then $A_{m+1}=\break y^m [0; a_{m+1}-1]$.

\item $y^{m+1}[0;a_{m+1}-1]=A_{m+1}$ and $A_i$ appear in $y^{m+1}$
syndetically with gaps bounded by $l_i$ for each $1\le i\le m+1$.

\item $F_{m+1}=\{i\in\Z_+:y^{m+1}(i)=1\}\subset F$.
\end{enumerate}

Set $a_{m+1}=u^m_1+b_m$ and let $A_{m+1}=y^m[0;a_{m+1}-1]$,
$B_{m+1}=A_{m+1}A_{m+1}0A_{m+1}$, and
$b_{m+1}=|B_{m+1}|=3a_{m+1}+1$. Then $A_{m+1}$ has the form of
$A_mV_mB_m$. Let $r_{m+1}=2l_m+2b_m+b_{m+1}$. As $F$ is thickly
syndetic, $1^{r_{m+1}}$ appears in $F$ at a syndetic set
$W_{m+1}=\{w^{m+1}_1,w^{m+1}_2,\ldots\}$. Without loss of generality
assume that $2r_{m+1}\le w_{j+1}^{m+1}-w_j^{m+1}\le
l_{m+1}-(l_m+b_m)$ and $2k_{m+1}\le w_1^{m+1}\le l_{m+1}-(l_m+b_m)$,
where $l_{m+1}$ is some number in $\N$.

To get $y^{m+1}$ we change $y^m$ at places $[w^{m+1}_i;w^{m+1}_i
+r_{m+1}-1]$ for each $i\in\N$. It is enough to show the idea how we
do at $[w^{m+1}_1;w^{m+1}_1+r_{m+1}-1]$.

Let $k,j$ satisfy that $u_{k-1}^m<w^{m+1}_1\le u_k^m$ and
$u_j^m+b_m-1\le w^{m+1}_1+r_{m+1}-1<u^m_{j+1}+b_m-1$.

Put $u_1^{m+1}=u_k^m+b_m$. Let $y^{m+1}[u_1^{m+1};
u_1^{m+1}+b_{m+1}-1]=B_{m+1}$ and
$$y^{m+1}[u_1^{m+1},u^m_j-1]=B_{m+1}(B_m)^{p_m}(B_{m-1})^{p_{m-1}}\ldots
(B_1)^{p_1}C_{m+1},$$ where $C_{m+1}$ is a word, and
$p_1,\ldots,p_m$ are natural numbers with

\begin{enumerate}
\item[$\bullet$] $|C_{m+1}|<b_1$,

\item[$\bullet$] $|C_{m+1}|+b_1p_1<b_2$, and

\item[$\bullet$] $|C_{m+1}|+b_1p_1+\ldots+b_ip_i<b_{i+1}$ for each $1\le i\le m-1$.
\end{enumerate}

That is, first we put $B_{m+1}$ at place $u_1^{m+1}$ and start from
$u_1^{m+1}+k_{m+1}$ to $u_j^m$ we put as many as $B_m$ we can and
then we put as many as $B_{m-1}$ we can and so on. We do the same at
all places $[w^{m+1}_i;w^{m+1}_i+r_{m+1}-1]$, we get $u_i^{m+1}\in
[w^{m+1}_i;w^{m+1}_i+r_{m+1}-1]$ with
$y^{m+1}[u_i^{m+1};u_i^{m+1}+b_{m+1}-1]=B_{m+1}$, $i=1,2\ldots$.

In such a way we get $y^{m+1}$. It is easy to see that $y^{m+1}$ and
$y^m$ differ possibly only at $[w^{m+1}_i;w^{m+1}_i+r_{m+1}-1]$,
$i=1,2,\ldots$. Thus
$$F_{m+1}=\{i\in\Z_+:y^{m+1}(i)=1\}\subset F_m\bigcup\cup_{i=1}^\infty
[w^{m+1}_i;w^{m+1}_i+r_{m+1}-1].$$ At the same time $B_i$ appears in
$y^{m+1}$ syndetically with gaps bounded by $l_i$ for each $1\le
i\le m+1$ by the construction and so is $A_i$ for each $1\le i\le
m+1$.
\medskip

In such a way for each $m\in\N$ we defined a finite word $A_m$. Let
$y=\lim A_m=\lim y^m$. By the construction, $A_m$ appears in $y$
with gaps bounded by $l_m$ for each $m\in\N$. That is, $y$ is a
minimal point for the shift. It is obvious that
$y\not=(0,0,\ldots)$. Let $Y=\overline{orb(y,\sigma)}$ and
$U=[1]=\{x\in Y: x(0)=1\}$. Then
$$\emptyset\not=N(y,U)=\bigcup_{i=1}^\infty\{i\in\Z_+: A_n(i)=1,
0\le i\le k_n-1\}\subset \bigcup_{i=1}^\infty F_n\subset F.$$ Thus
$F$ contains the m-set $N(y,U)$. The proof is completed.
\hfill
$\square$

\begin{rem}\label{wmpoint}
In fact, in the proof of Proposition \ref{tsyn}, $(Y,\sigma)$ is a
weakly mixing system. Indeed, for each $m\in\N$, $A_{m+1}$ has the
form $A_mV_mB_m$ i.e. the form $A_mV_mA_mA_m0A_m$, so we know that
$N([{A_m}],[{A_m}])=N(y,[{A_m}])-N(y,[{A_m}])\supset \{a_m,a_m+1\}$
which implies that $Y$ is weakly mixing (see Lemma \ref{wmixing}
below).
\end{rem}


\subsection{Condition in \cite{HO}}

In this subsection we will show that there is no minimal t.d.s.
satisfying the sufficient condition in \cite[Theorem 3.1]{HO}. Let
$(X, T)$ be a t.d.s.. Say $x\in X$ satisfies the property ($\star$):

\medskip

\noindent ($\star$)\quad {\em if for each neighborhood $V$ of $x$,
there exists $n=n(V)$ such that if $S\subset \Z_+$ is a finite
subset with $|s-t|\ge n$ for all distinct $s,t\in S$, then there
exists $\ell\in \Z_+$ such that $T^{s+\ell}x\in V$ for all $s\in
S$.}

\medskip

We will show that if $(X,T)$ is a transitive system with a
transitive point $x$ satisfying ($\star$) then it is weakly mixing.
Note that the orbit closure of an $\F_s$-PR point needs not to be
weakly mixing (see Remark \ref{rem4.7}).

\medskip

First we need the following lemma.

\begin{lem}\cite[Lemma 5.1]{HY}\label{wmixing}
Let $(X,T)$ be a transitive system. If for any open non-empty subset
$U$ of $X$ there is $s=s_U\in \Z_+$ such that $s,s+1\in N(U,U)$,
then $(X, T)$ is weakly mixing.
\end{lem}

Let $\F_{rs}$ be the smallest family containing $\{n\Z_+: n\in\N\}$.
The following notion was introduced in \cite{HY}. Let $(X,T)$ be a
t.d.s.. We say $(X,T)$ has {\it dense small periodic sets}, if for
any open and non-empty subset $U$ of $X$ there exists a non-empty
closed $A\subset U$ and $k\in\N$ such that $A$ is invariant for
$T^k$. Now we are ready to show

\begin{lem}\label{easy}
Let $(X, T)$ be a transitive system with a transitive point $x$
satisfying ($\star$). Then $(X,T)$ is weakly mixing, and it has
dense small periodic sets.
\end{lem}

\begin{proof}
First we show $(X,T)$ is weakly mixing. Let $U$ be a non-empty open
subset of $X$ and $V$ be a neighborhood of $x$ such that $T^m
V\subset U$ for some $m\in\N$. Assume $n=n(V)$ is the number
appearing in the definition of ($\star$). Then there is $\ell\in
\Z_+$ such that $\{\ell+n,\ell+2n, \ell+3n+1\}\subset N(x,V)$. That
is, $T^{\ell+n}x, T^{\ell+2n}x, T^{\ell+3n+1}x\in V$, which implies
that $T^{m+\ell+n}x, T^{m+\ell+2n}x, T^{m+\ell+3n+1}x\in U$. Thus
$\ell+n, \ell+2n,\ell+3n+1\in N(T^mx,U)$. We have
$$N(U,U)=N(T^mx,U)-N(T^mx,U)\supset \{n,n+1\}.$$ By Lemma
\ref{wmixing}, $(X,T)$ is weakly mixing.

Now we show $(X,T)$ has dense small periodic sets. Let $V$ be a
neighborhood of $x$ and $n=n(V)$ be the number appearing in the
definition of ($\star$). By ($\star$) for all $k\in\Z_+$ there is
some $l=l(k)\in \Z_+$ such that $\bigcap_{j=0}^k
T^{-jn-l}V\not=\emptyset$. That is, $\bigcap_{j=0}^k
T^{-jn}V\not=\emptyset$. By a compactness argument we have
$\bigcap_{j=0}^\infty T^{-jn}{\overline{V}}\not=\emptyset$. This
implies that there is $y\in\bigcap_{j=0}^\infty
T^{-jn}{\overline{V}}$ such that $T^{jn}y\in \overline{V}$ for all
$j\in\Z_+$. Thus, $(X,T)$ has dense small periodic sets since $x$ is
transitive.
\end{proof}

With the help of Lemma \ref{easy} we have

\begin{thm}
There is no minimal t.d.s. with points satisfying ($\star$).
\end{thm}

\begin{proof}
Assume the contrary that there is a minimal t.d.s. $(X,T)$ with
points satisfying ($\star$). Then on the one hand, by Lemma
\ref{easy} $(X,T)$ has dense small periodic sets, and hence $(X,T)$
is not totally transitive. But on the other hand, also by Lemma
\ref{easy}, $(X,T)$ is totally minimal, a contradiction.
\end{proof}

\section{$\F$-product recurrence for zero entropy}
Entropy is a measurement of complexity or chaos of a t.d.s.. For a
t.d.s. $(X,T)$ the entropy of $(X,T)$ will be denoted by $h(T)$. For
the definitions and basic properties of entropy and how to compute
the entropy of a symbolic system we refer to \cite{Walters}. In this
section we investigate the properties of points whose product with
points whose orbit closure having zero entropy is recurrent. We show
if $(x,y)$ is recurrent for any point $y$ whose orbit closure is a
minimal system having zero entropy, then $x$ is
$\F_{pubd}$-recurrent, and if $(x,y)$ is recurrent for any point $y$
whose orbit closure is an $M$-system having zero entropy, then $x$
is minimal. Moreover, it turns out that if $(x,y)$ is recurrent for
any recurrent $y$ whose orbit closure has zero entropy, then $x$ is
distal.

\subsection{$\F$-PR$_0$}
\begin{de}
Let $(X,T)$ be a t.d.s. and $\F$ be a family. $x\in X$ is
$\F$-PR$_0$ if for any t.d.s. $(Y,S)$ and any $\F$-recurrent point
$y\in Y$ whose orbit closure $\overline{orb(y,S)}$ having zero
entropy, $(x,y)$ is a recurrent point of $(X\times Y, T\times S)$.
\end{de}

It is cleat that
\begin{equation*}
\xymatrix { &{\F_{inf}\!-\!PR} \ar[d]\ar[r]& {\F_{pubd}-PR }
\ar[d]\ar[r]&
{\F_{ps}-PR}\ar[d]\ar[r]&  {\F_s-PR}\ar[d] \\
&{\F_{inf}\!-\!PR_0} \ar[r]&  {\F_{pubd}-PR_0 } \ar[r]&
{\F_{ps}-PR_0} \ar[r]&  {\F_s-PR_0}}
\end{equation*}
Where ``$\longrightarrow$'' means implication.


Recall that $x$ is $\F_{inf}$-PR if and only if $x$ is distal. We
have

\begin{thm}\label{thm5.2}
Let $(X,T)$ be a t.d.s. and $x\in X$. Then $x$ is $\F_{inf}$-PR$_0$
if and only if it is distal.
\end{thm}

\begin{proof}
If $x$ is distal, then it is clear that it is $\F_{inf}$-PR$_0$. Now
assume that $x$ is $\F_{inf}$-PR$_0$. Let $A$ be an IP-set. Then $A$
contains a sub IP-set $B$ with zero entropy (see for example
\cite{HKY}). Then $N(x,U)$ is $IP^*$ for each neighborhood $U$, and
$x$ is distal by Proposition \ref{distal}.
\end{proof}

Similar to Theorem \ref{disjoint} we have

\begin{thm}\label{disjoint0}
Let $(X,T)$ be a transitive t.d.s. which is disjoint
from any minimal system with zero entropy. Then each point in
$Tran_T$ is $\F_s$-PR$_0$.
\end{thm}

It was proved in \cite{B2} that a transitive diagonal system is
disjoint from all minimal t.d.s. with zero entropy. Thus if $(X, T)$
is a transitive diagonal t.d.s. then each transitive point $x$ is in
$\F_s$-PR$_0$. It was proved in \cite{HKY} that every subset of
$\Z_+$ with lower Banach density 1 contains an m-set $A$ such that
the orbit closure of $1_A$ has zero entropy. With a small
modification we have the following proposition.

\begin{prop}\label{tsynE}
Every subset of $\Z_+$ with lower Banach density 1 containing
$\{0\}$ contains an sm-set $A$ such that the orbit closure of $1_A$
has zero entropy.
\end{prop}

Using the same argument as in Theorem \ref{orbitM} we have

\begin{thm}\label{orbitE} The orbit closure of an $\F_s$-PR$_0$ point is an
$E$-system.
\end{thm}
\begin{proof} Let $x$ be an $\F_s$-PR$_0$ point and $U$ be an open
neighborhood of $x$. If $N(x,U)$ has zero Banach density, then the
lower Banach density of $A=\Z_+\setminus N(x,U)$ is 1. Then by
Proposition \ref{tsynE}, $A\cup\{0\}$ contains $N(y, V)$, where $(Y,
S)$ is a minimal set, $y\in Y$, $V$ is an open neighborhood of $y$
and $h(S)=0$. Thus we have $N((x,y), U\times V)=N(x, U)\cap N(y,
V)\subset \{0\}$, a contradiction.
\end{proof}

Let $E(X,T)$ be the set of all entropy pairs (see \cite{B2}). A
t.d.s. $(X,T)$ is {\it diagonal} if $\{(x,Tx):x\in X\}\subset
E(X,T)$ and {\it u.p.e.} if $E(X,T)=X^2\setminus \Delta$. In
\cite{HKY} a transitive diagonal t.d.s. with a unique minimal point
was constructed (see \cite{HLY} for more examples). Thus we have
$$\F_s-PR_0\not\Rightarrow \F_s-PR.$$

We remark that there is a minimal point $x$ which is $\F_s$-PR$_0$
and is not $\F_s$-PR. In fact by \cite{BHR} if $h(T)>0$ then there
are asymptotic pairs $(x,y)$ with $x\not=y$, and by \cite{GW} or
\cite{HY2} there are minimal u.p.e. systems.

\subsection{$\F_{ps}$-PR$_0$}

In Theorem \ref{thickM} we have shown that if a point $x$ is
$\F_{ps}$-PR, then $x$ is minimal. Here is a natural question: if
$x$ is $\F_{ps}$-PR$_0$, is $x$ minimal? The answer is affirmative.
That is, we have

\begin{thm}\label{thick0}
Let $(X,T)$ be a t.d.s.. If $x\in X$ is $\F_{ps}$-PR$_0$, then it is
minimal.
\end{thm}
\begin{proof} According to the proof of Theorem \ref{thickM} it
remains to show that the point $y$ constructed in Proposition
\ref{thick} has zero entropy.

Recall that
$$A_{k+1}=A_k0^{a_k}A_kA_{k-1}^{n_{k+1}^{k-1}}\ldots
A_2^{n^2_{k+1}}A_1^{n^1_{k+1}}\le (x_0, \ldots, x_{l_{k+1}})$$ with
$|A_{1}|^{n^1_{k+1}}=|A_2|^{n^2_{k+1}}=\ldots=|A_{k-1}|^{n_{k+1}^{k-1}}=|A_k|$,
$a_{k+1}$ can be divided by $|A_{k}|$ and $y=\lim_{k\lra\infty}A_k$.
Let $X=\overline{orb(y,\sigma)}$ and $m_{k}=|A_k|$.

We are going to show that $h(X, \sigma)=0$. Let $$B_k(y)=\# \{ u\in
\{0,1\}^k: \exists i\in \Z_+\ \text{such that}\ u=y[i;i+k-1] \},$$
where $\# (\cdot)$ means the cardinality of a set. Then $h(X,
\sigma)=\lim_{k\rightarrow\infty}\frac{1}{m_k}\log B_{m_{k}}(y)$.
Let $u\in \{0,1\}^{m_{k}}$ appear in $y$. Then there exists $i>k$
such that $u$ appears in $A_i$. By the way of the construction of
$A_{j}, j\in\mathbb{N}$, it is known that $A_{i}=W_{0}W_{1}\cdots
W_{s}$, where $W_{j}$ has the form of $0^{m_{k}} A_{k}
A_{k-1}^{n_{k+1}^{k-1}} \ldots A_2^{n^2_{k+1}} A_1^{n^1_{k+1}}$ with
$|0^{m_{k}}|=|A_{k}|=|A_{k-1}^{n_{k+1}^{k-1}}|=\ldots=|A_2^{n^2_{k+1}}|=|A_1^{n^1_{k+1}}|$.
So we have that $$B_{m_{k}}(y)\le (m_{k}+1)(k+1)k\le (m_{k}+1)^3.$$
It follows that
$$h(X, \sigma)=\lim_{k\rightarrow\infty}\frac{1}{m_k}\log
B_{m_{k}}(y)=0.$$ This ends the proof.
\end{proof}

\subsection{Summary and some questions}

Let ${\mathfrak{E}}_0$ be the collection of all $E$-systems with
zero entropy, and ${\mathfrak{M}}_0$ be the collection of all
$M$-systems with zero entropy. The following proposition is from
\cite{HKY}. Recall that a t.d.s. $(X,T)$ is {\it c.p.e.} if the
factor induced by the smallest closed invariant equivalence relation
containing $E(X,T)$ is trivial.

\begin{prop} \label{thm-1-4} The following statements hold.

\begin{enumerate}
\item If $X\perp {\mathfrak{E}}_0$ (i.e. $X$ is disjoint from each element
of ${\mathfrak{E}}_0$), then $X$ is minimal and has c.p.e..

\item If $X$ is minimal and for each $\mu\in M(X,T)$,
$(X,\mathcal{B}_X,\mu,T)$ is a measurable K-system, then $X\perp
{\mathfrak{E}}_0$.

\item If $X$ is a minimal diagonal system then $X\perp {\mathfrak{M}}_0$.
\end{enumerate}
\end{prop}

Thus we have
\begin{thm} The following statements hold.

\begin{enumerate}
\item $\F_{pubd}-PR_0\not\Rightarrow \F_{inf}-PR_0$.


\item $\F_{ps}-PR_0\not\Rightarrow \F_{ps}-PR$.

\item $\F_{pubd}-PR_0\not\Rightarrow \F_{pubd}-PR.$
\end{enumerate}

\end{thm}
\begin{proof}
(1) Let $(X, T)$ be a minimal t.d.s. such that there is $\mu\in
M(X,T)$ with $(X,\mathcal{B}_X,\mu,T)$ being a measurable K-system.
Then each point of $X$ is in $\F_{pubd}$-$PR_0$. Since in such a
system, there exists asymptotic pairs, we have
$\F_{pubd}-PR_0\not\Rightarrow \F_{inf}-PR_0$.

\medskip

(2) and (3) follow from Proposition \ref{thm-1-4}.
\end{proof}

The following question is open:
$$\F_{ps}-PR_0\not\Rightarrow \F_{pubd}-PR_0?$$
Note that it is an open question if there is a t.d.s. in $M_0^\perp
\setminus E_0^\perp$, see \cite{HKY}.

To sum up we have

\begin{equation*}
\xymatrix { &{\F_{inf}\!-\!PR}\ar[d] & {\F_{pubd}-PR } \ar[l]_{?}&
{\F_{ps}-PR}\ar[l]_?&  {\F_s-PR}\ar[l]_{not} \\
&{\F_{inf}\!-\!PR_0}\ar[u] &  {\F_{pubd}-PR_0 }\ar[l]_{not}
\ar[u]_{not}& {\F_{ps}-PR_0}\ar[u]_{not}\ar[l]_{?}& {\F_s-PR_0}
\ar[u]_{not}\ar[l]_{not}}
\end{equation*}


For minimal systems we have

\begin{equation*}
\xymatrix { &{\F_{inf}\!-\!PR}\ar[d] & {\F_{pubd}-PR } \ar[l]_{?}&
{\F_{ps}-PR}\ar[l]_?&  {\F_s-PR}\ar[l]_{?} \\
&{\F_{inf}\!-\!PR_0}\ar[u] &  {\F_{pubd}-PR_0 }\ar[l]_{not}
\ar[u]_{not}& {\F_{ps}-PR_0}\ar[u]_{not}\ar[l]_{?}& {\F_s-PR_0}
\ar[u]_{not}\ar[l]_{?}}
\end{equation*}

\section{Factors and extensions}

In this section we investigate product recurrent properties for a
family under factors or extensions. In this section and the next
section we will use some tools from the theory of Ellis semigroup,
see \cite{Au, G, V77, Vr} for details.

\subsection{Definitions on factors}
A {\em homomorphism} $\pi: X\rightarrow Y$ between the t.d.s.
$(X,T)$ and $(Y,S)$ is a continuous onto map which intertwines the
actions; one says that $(Y,S)$ is a {\it factor} of $(X,T)$ and that
$(X,T)$ is an {\em extension} of $(Y,S)$, and one also refers to
$\pi$ as a {\em factor map} or an {\em extension}. The systems are
said to be {\em conjugate} if $\pi$ is bijective. An extension $\pi$
is determined by the corresponding closed invariant equivalence
relation $R_{\pi} = \{ (x_1,x_2): \pi x_1= \pi x_2 \} =(\pi \times
\pi )^{-1} \Delta_Y \subset  X \times X$.

An extension $\pi : (X,T) \rightarrow (Y,S)$ is called {\em
proximal} if $R_{\pi} \subset P(X,T)$. Similarly we define {\em
distal} extensions. An extension $\pi$ is {\em equicontinuous} if
for every $\epsilon >0$ there is $\delta >0$ such that $(x,y) \in
R_{\pi}$ and $d(x,y)<\delta$ implies $d(T^nx,T^ny)<\epsilon$, for
every $n \in \N$. And $\pi$ is called {\em almost one-to-one} if the
set $X_0=\{x \in X: \pi^{-1}(\pi(x)) = \{x\}\}$ is a dense
$G_\delta$ subset of $X$.

\subsection{Product recurrent properties under factors or extensions}
In this subsection we will use the following basic result
frequently: $x$ is recurrent if and only if there is an idempotent
$u$ such that $ux=x$ (please refer to \cite{A, AAG, EEN, F1} etc.
for details).

\begin{prop}\label{rlift}
Let $\pi:X\lra Y$ be a factor map. If $x\in R(X, T)$ then $\pi(x)\in
R(Y,S)$. Conversely, if $y\in R(Y, S)$ then there is $x\in
\pi^{-1}(y)\cap R(X,T)$.
\end{prop}

\begin{proof}
Let $y\in R(Y,S)$. Then there is an idempotent $u$ with $uy=y$. Take
$x'\in \pi^{-1}(y)$ and set $x=ux'$. Then $x\in R(X,T)$ and $\pi
(ux')=u\pi (x')=y$.
\end{proof}

\begin{cor}
Let $(Y,S)$ be a t.d.s and $y\in Y$ be recurrent. Then for any
t.d.s. $(X,T)$, there is $x\in X$ such that $(x,y)$ recurrent.
\end{cor}

\begin{proof}
One can get the corollary from Proposition \ref{rlift} or
Proposition \ref{prop2.4}.
\end{proof}

Let $\pi: (X,T)\rightarrow (Y,S)$ be a factor map. Recall a point
$x\in X$ is called {\em $\pi$-distal} if $(x',x)\in P(X,T)$ and
$\pi(x')=\pi(x)$ then $x=x'$.

\begin{thm}\label{thm6.3}
Let $\F$ be a family, $(X,T), (Y,S)$ be two t.d.s. and $\pi:X\lra Y$
be a factor map.
\begin{enumerate}
  \item If $x$ is $\F$-PR, then $\pi(x)$ is $\F$-PR.

  \item If $x$ is $\pi$-distal and $y=\pi(x)$ is $\F$-PR, then $x$
  is $\F$-PR.

  \item If $y\in Y$ satisfies $\pi^{-1}(y)=\{x\}$ for some $x\in X$ and $y$ is
$\F$-PR, then $x$ is $\F$-PR.
\end{enumerate}
\end{thm}

\begin{proof}
(1) Let $x$ be $\F$-PR and $X_1$ be the orbit closure of $x$. Assume
that $z$ is a $\F$-recurrent point and $Z=\overline{orb(z)}$. Then
$\pi\times Id:X_1\times Z\lra Y\times Z$ is a factor map. Since $x$
is $\F$-PR, $(x,z)$ is a recurrent point. It follows that
$(\pi(x),z)$ is a recurrent point, and thus $\pi(x)$ is $\F$-PR.

(2) Assume $y$ is $\F$-PR. Let $z$ be a $\F$-recurrent point. Then
$(y,z)$ is recurrent, and hence there exists an idempotent $u$ such
that $u(y,z)=(y,z)$. Now we have $\pi(ux)=u\pi(x)=uy=y=\pi(x)$ and
note that $(x,ux)\in P(X,T)$. Since $x$ is $\pi$-distal, we have
$ux=x$. Thus $u(x,z)=(x,z)$, i.e. $(x,z)$ is recurrent. Hence $x$ is
$\F$-PR.

(3) is a special case of (2).
\end{proof}


\begin{thm}\label{min-times}
Let $(X,T), (Y,S)$ be t.d.s.
\begin{enumerate}
  \item If $(X,T)$ and $(Y,S)$ have dense sets of minimal
points (resp. $E$-systems, $P$-systems), then so does $X\times Y$.

  \item If $(X,T)$ has a measure with full support and $(Y,S)$ has a dense
set of recurrent points, then $X\times Y$ has a dense set of
recurrent points.

  \item There are transitive t.d.s. $(X,T)$ and $(Y,S)$ such that $X\times
Y$ does not have a dense set of recurrent points.
\end{enumerate}
\end{thm}

\begin{proof}
If $(X,T)$ and $(Y,S)$ have dense sets of periodic points, or have
measures with full support, then it is clear that so does $(X\times
Y, T\times S)$.

If $X$ and $Y$ are minimal then there is a minimal point $(x,y)\in
X\times Y$. Since $T^n\times S^m:X\times Y\lra X\times Y$ is a
factor map it follows that $(T^nx,S^my)$ is minimal for each pair
$(n,m)\in \Z_+\times \Z_+$. Thus the set of minimal points in
$X\times Y$ is dense. This implies that if $X$ and $Y$ have dense
sets of minimal points then so does $X\times Y$.

\medskip
Now assume that $X$ has a measure with full support and $Y$ has a
dense set of recurrent points. Without loss of generality we assume
that $X$ is an $E$-system and $Y$ is transitive. For non-empty open
sets $U\subset X$ and $V\subset Y$, pick transitive points $x\in U$
and $y\in V$. Then
$$N(U\times V, U\times V)=N(U,U)\cap N(V,V)=(N(x,U)-N(x,U))\cap
(N(y,V)-N(y,V)).$$ Since $N(x,U)\in \F_{pubd}$, $N(x,U)-N(x,U)$ is
an $IP^*$-set \cite[Theorem 3.18.]{F1}. This implies that $N(U\times
V, U\times V)$ is infinite. That is $X\times Y$ is non-wandering
which implies that the set of recurrent points in $X\times Y$ is
dense \cite[Theorem 1.27.]{F1}.

\medskip
Let $F_1$ and $F_2$ be two disjoint thick sets. Let $A_1$ and $A_2$
be two IP-sets contained in $F_1$ and $F_2$ respectively. Moreover
we may assume that $A_i$ is generated by $\{p^i_j\}$ with
$$p^i_{j+1}>p^i_1+\ldots + p^i_j$$ for all
$j\in \N$ and  $A_i-A_i\subset F_i$ for $i=1,2$. Let
$X_i=\overline{orb(1_{A_i}, \sigma)}\subseteq \{0,1\}^{\Z_+}$,
$i=1,2$. Then $X_1\times X_2$ does not have a dense set of recurrent
points, since
\begin{align*}
N([1]_X\times [1]_Y,[1]_X\times [1]_Y)&=N([1]_X,[1]_X)\cap
N([1]_Y,[1]_Y)\\
&=(A_1-A_1)\cap (A_2-A_2)\subset F_1\cap F_2=\emptyset.
\end{align*}
\end{proof}


\section{Disjointness and weak disjointness}

Let $\mathcal{T}$ be a class of t.d.s. and $(X,T)$ be a t.d.s. If
$(X,T)\perp (Y,S), \forall (Y,S)\in \mathcal{T}$, then we denote it
by $(X,T)\perp \mathcal{T}$ or $(X,T)\in \mathcal{T}^\perp$, where
$\mathcal{T}^\perp=\{(X,T): (X,T)\perp \mathcal{T}\}$.

Let $\M$ be the class of all minimal systems and $\M_0$ be the class
of all minimal systems with zero entropy. Let $\M_{eq}$ (resp.
$\M_d$ and $\M_{wm}$) be the class of all minimal equicontinuous
(resp. distal and weakly mixing) systems. In \cite{F}, Furstenberg
asked the question: {\em Describe the classes $\M^\perp$ and
${\M_d}^\perp$.} We extend the question:
\begin{ques}
Which t.d.s. is disjoint from $\M$, $\M_0$, $\M_{eq}$, $\M_d$ and
$\M_{wm}$? Or determine $\M^\perp$, $\M_0^\perp$, $\M_{eq}^\perp$,
$\M_d^\perp$ and $\M_{wm}^\perp$.
\end{ques}
A related question is about the weak disjointness. In this section
we will summarize what one knows concerning the above question and
give additional new results.

\subsection{Some basic properties on disjointness}\label{sec7.1}

Let $\pi: (X,T)\rightarrow (Y,S)$ be an extension between two t.d.s.
$(X,T)$ and $(Y,S)$. $\pi$ is called {\em minimal} if the only
closed invariant subset $K$ of $X$ such that $\pi(K)=Y$ is $X$
itself. Clearly, $X$ is minimal if and only if $\pi$ is minimal and
$Y$ is minimal. More generally, let $\pi: X\rightarrow Y$, $\psi:
Y\rightarrow Z$ be extensions, then $\psi\circ \pi$ is a minimal
extension if and only if both $\psi$ and $\pi$ are minimal
extensions.

By definitions it is easy to get the following important
observation:
\begin{lem}\label{minimal-extension}
Let $(X,T)$ be a t.d.s. and let $(Y,S)$ be minimal. Then $(X,T)\perp
(Y,S)$ if and only if the projection map $\pi_1: X\times
Y\rightarrow X$ is a minimal extension.
\end{lem}

An extension $\pi: X\rightarrow Y$ is said to be {\em semi-distal}
if $(x,y)\in R_\pi$ is both recurrent and proximal, then $x=y$.

\begin{lem}\cite[Theorem 2.14.]{AAG}\label{semi-distal}
Let $\pi: (X,T)\rightarrow (Y,S)$ be a factor map. If $X$ is
transitive and $\pi$ is semi-distal, then $\pi$ is minimal.
\end{lem}

Since each equicontinuous or distal extension is semi-distal, we
have
\begin{cor}
Let $\pi: (X,T)\rightarrow (Y,S)$ be a factor map. If $X$ is
transitive and $\pi$ is equicontinuous or distal, then $\pi$ is
minimal.
\end{cor}

The following proposition concerns the `lifting' of disjointness by
semi-distal extensions.

\begin{prop}\label{disjoint-equ}
Let $(X,T)$ be a t.d.s. and $\pi: (Y',S')\rightarrow (Y,S)$ be an
extension of minimal systems. If $\pi$ is semi-distal (resp. distal,
equicontinuous) and $(X\times Y', T\times S')$ is transitive, then
\begin{equation*}
    X\perp Y' \quad \text{ if and only if }\quad X\perp Y.
\end{equation*}
\end{prop}

\begin{proof}
It follows from Lemma \ref{minimal-extension} and Lemma
\ref{semi-distal}.
\end{proof}

The following proposition concerns the 'lifting' of disjointness by
proximal extensions.
\begin{lem}\label{proximal-extension}
Let $\pi: (X,T)\rightarrow (Y,S)$ be an extension. If $X$ has a
dense set of minimal points and $\pi$ is proximal, then $\pi$ is
minimal.
\end{lem}

\begin{proof}
Let $J$ be a closed invariant subset of $X$ with $\pi(J)=Y$. Let $x$
be a minimal point of $X$. Since $\pi(J)=Y$, there is $x'\in J$ such
that $\pi(x)=\pi(x')$. Now as $\pi$ is proximal, $x,x'$ are
proximal. Hence by minimality of $x$, $$x\in
\overline{orb(x,T)}\subset J.$$ Since the set of minimal points of
$X$ is dense, $J=X$. That is, $\pi$ is minimal.
\end{proof}

\begin{prop}\label{disjoint-proximal}
Let $(X,T)$ be a t.d.s. and $\pi: (Y',S')\rightarrow (Y,S)$ be an
extension of minimal systems. If $\pi$ is proximal and $(X\times Y',
T\times S')$ has a dense set of minimal points, then
\begin{equation*}
    X\perp Y' \quad \text{ if and only if }\quad X\perp Y.
\end{equation*}
\end{prop}

\begin{proof}
It follows from Lemma \ref{minimal-extension} and Lemma
\ref{proximal-extension}.
\end{proof}

Finally, we have the following property:

\begin{prop}\cite{AG}\label{inverse limit}
Disjointness is a residual property, i.e. it is inherited by
factors, irreducible lifts and inverse limits.
\end{prop}

\subsection{A note on $\Z_+$-actions and $\Z$-actions}\label{sec7.2}
In the sequel we will deal with the structure theorem of minimal
systems. This theory was mainly developed for group actions and
accordingly we assume that $T$ is a homeomorphism when we use the
related results.

To get the results for surjective maps we need to consider the
natural extensions. For a t.d.s. $(X,T)$ with a metric d, we say
$(\widetilde{X},\widetilde{T})$ is the {\em natural extension} of
$(X,T)$, if $\widetilde{X}=\{ (x_1,x_2,\cdots): T(x_{i+1})=x_i,
x_i\in X, i\in \N \}$, which is a subspace of the product space
$\Pi_{i=1}^\infty X$ with the compatible metric $d_T$ defined by
$d_T((x_1,x_2,\cdots),(y_1,y_2,\cdots))=\sum_{i=1}^\infty
\frac{d(x_i,y_i)}{2^i}.$ Moreover, $\widetilde{T}:\widetilde{X}\lra
\widetilde{X}$ is the shift homeomorphism, i.e.
$\widetilde{T}(x_1,x_2,\cdots)=(T(x_1),x_1,x_2,\cdots)$.
The important fact is that: $(X,T)\perp (Y,S)$ if and only if
$(\widetilde{X},\widetilde{T}) \perp (\widetilde{Y},\widetilde{S})$,
where $(\widetilde{X},\widetilde{T})$ and
$(\widetilde{Y},\widetilde{S})$ are the natural extensions of
$(X,T)$ and $(Y,S)$ respectively \cite[Proposition 1.1.]{HY}. Hence
when considering disjointness of two systems, we can can assume both
of them are homeomorphisms.

Another problem is that the traditional structure theory of minimal
systems is developed for group actions, and that means here it works
for $\Z$-actions. But till now we only confront $\Z_+$-actions. This
is not a big problem here, since by definition it is easy to verify
that for two homeomorphism systems they are disjoint under the
$\Z_+$-actions if and only if they do under the $\Z$-actions. Note
that when we consider $\Z$-actions, the notions defined before are a
little different. For example, for $\Z$-actions $(x,y)$ of $X$ is
{\it proximal} if there is a subsequence $\{n_i\}$ in $\Z$ such that
$\lim_{n\ra \infty} T^{n_i} x =\lim_{n\ra \infty} T^{ n_i} y$. We
deal with other notions in the similar way. It is easy to check that
all results of Subsection \ref{sec7.1} still hold when considering
$\Z$-actions.

To sum up, in the sequel when we deal with the structure theorem of
minimal systems, we assume that $T$ is a homeomorphism and use
related results freely.

\subsection{Structure theorem for minimal systems}

In this subsection we briefly review the structure theorem of
minimal systems.

We say that a minimal system $(X,T)$ is a {\em strictly PI system}
if there is an ordinal $\eta$ (which is countable when $X$ is
metrizable) and a family of systems
$\{(W_\iota,w_\iota)\}_{\iota\le\eta}$ such that (i) $W_0$ is the
trivial system, (ii) for every $\iota<\eta$ there exists a
homomorphism $\phi_\iota:W_{\iota+1}\to W_\iota$ which is either
proximal or equicontinuous, (iii) for a limit ordinal $\nu\le\eta$
the system $W_\nu$ is the inverse limit of the systems
$\{W_\iota\}_{\iota<\nu}$,  and (iv) $W_\eta=X$. We say that $(X,T)$
is a {\em PI-system} if there exists a strictly PI system $\tilde X$
and a proximal homomorphism $\theta:\tilde X\to X$.

If in the definition of PI-systems we replace proximal extensions by
almost one-to-one extensions we get the notion of HPI {\em systems}.
If we replace the proximal extensions by trivial extensions (i.e.\
we do not allow proximal extensions at all) we have I {\em systems}.
These notions can be easily relativized and we then speak about I,
HPI, and PI extensions.

We have the following structure theorem for minimal systems, for
details see \cite{Au, EGS, G, V77, Vr} etc..

\begin{thm}[Structure theorem for minimal systems]\label{structure}
Given a homomorphism $\pi: X \to Y$ of minimal dynamical system,
there exists an ordinal $\eta$ (countable when $X$ is metrizable)
and a canonically defined commutative diagram (the canonical
PI-Tower)
\begin{equation*}
\xymatrix
        {X \ar[d]_{\pi}             &
     X_0 \ar[l]_{{\theta}^*_0}
         \ar[d]_{\pi_0}
         \ar[dr]^{\sigma_1}         & &
     X_1 \ar[ll]_{{\theta}^*_1}
         \ar[d]_{\pi_1}
         \ar@{}[r]|{\cdots}         &
     X_{\nu}
         \ar[d]_{\pi_{\nu}}
         \ar[dr]^{\sigma_{\nu+1}}       & &
     X_{\nu+1}
         \ar[d]_{\pi_{\nu+1}}
         \ar[ll]_{{\theta}^*_{\nu+1}}
         \ar@{}[r]|{\cdots}         &
     X_{\eta}=X_{\infty}
         \ar[d]_{\pi_{\infty}}          \\
        Y                 &
     Y_0 \ar[l]^{\theta_0}          &
     Z_1 \ar[l]^{\rho_1}            &
     Y_1 \ar[l]^{\theta_1}
         \ar@{}[r]|{\cdots}         &
     Y_{\nu}                &
     Z_{\nu+1}
         \ar[l]^{\rho_{\nu+1}}          &
     Y_{\nu+1}
         \ar[l]^{\theta_{\nu+1}}
         \ar@{}[r]|{\cdots}         &
     Y_{\eta}=Y_{\infty}
    }
\end{equation*}
where for each $\nu\le\eta, \pi_{\nu}$ is RIC, $\rho_{\nu}$ is
isometric, $\theta_{\nu}, {\theta}^*_{\nu}$ are proximal and
$\pi_{\infty}$ is RIC and weakly mixing of all orders. For a limit
ordinal $\nu ,\  X_{\nu}, Y_{\nu}, \pi_{\nu}$ etc. are the inverse
limits (or joins) of $ X_{\iota}, Y_{\iota}, \pi_{\iota}$ etc. for
$\iota < \nu$.

Thus if $Y$ is trivial, then $X_\infty$ is a proximal extension of
$X$ and a RIC weakly mixing extension of the strictly PI-system
$Y_\infty$. The homomorphism $\pi_\infty$ is an isomorphism (so that
$X_\infty=Y_\infty$) if and only if  $X$ is a PI-system.
\end{thm}

Reall an extension $\pi : X \to Y$ of minimal systems {\em
relatively incontractible (RIC) extension}\ if it is open and for
every $n \ge 1$ the minimal points are dense in the relation
$$
R^n_\pi = \{(x_1,\dots,x_n) \in X^n : \pi(x_i)=\pi(x_j),\ \forall \
1\le i \le j \le n\}.
$$

\subsection{Disjointness for $\M_{pi}$}
In this subsection we discuss disjointness for $\M_{pi}$ which is
the collection of all minimal PI-systems. It is known \cite{F} that
$\M_{eq}^\perp\cap \M=\M_{wm}$ which implies that $\M_{pi}^\perp\cap
\M=\M_{wm}$ (see Theorem \ref{pid}). In this subsection we will show
that each weakly mixing t.d.s. with dense minimal points is disjoint
from all minimal PI-systems. We remark that a weakly mixing t.d.s.
(even scattering) is disjoint from all HPI minimal t.d.s. (using
Propositions \ref{disjoint-equ} and \ref{inverse limit}).

\begin{thm}\label{pid}
Each weakly mixing t.d.s. with dense minimal points is disjoint from
all minimal PI-systems.
\end{thm}

\begin{proof}
Since a PI system is constructed by equicontinuous and proximal
extensions, the result follows from Propositions \ref{disjoint-equ},
\ref{disjoint-proximal} and \ref{inverse limit} and the well known
facts: \begin{enumerate}
\item[$\bullet$] a weakly mixing t.d.s., is weakly disjoint from all
minimal t.d.s. \cite{BHM}, (since a weakly mixing t.d.s. is
scattering).

\item[$\bullet$] the product of two systems with dense sets of
minimal points still have a dense set of minimal points (Theorem
\ref{min-times}).

\item[$\bullet$] a weakly mixing t.d.s. is  disjoint from all
minimal equicontinuous t.d.s. \cite{F}.

\end{enumerate}
\end{proof}

\begin{rem}
Note that a weakly mixing system with dense minimal points is not
necessarily disjoint from all minimal systems. Let $(X,T)$ be a
minimal weakly mixing t.d.s. and $(Y,S)=(X\times X, T\times T)$.
Then $(Y, S)$ is weakly mixing and has a dense set of minimal
points. We claim that $(Y,S)\not\perp (X,T)$. In fact
$J=\{(x,y,x):x,y\in X\}$ is a joining and it is clear that
$J\not=X\times X\times X$.
\end{rem}

\begin{rem}\label{rem-PI}
By the structure theorem of a minimal t.d.s. and the result in
\cite{HY} to obtain the necessary and sufficient condition for
disjointness from all minimal t.d.s. (for a transitive t.d.s.) is
equivalent to find such a condition (implying weakly mixing, dense
minimal points and something more) such that if $X$ satisfies the
condition, and $X$ is disjoint from a minimal t.d.s. $Y'$, then $X$
is disjoint from all minimal t.d.s. $Y$ satisfying that $\pi:Y\ra
Y'$ is a weakly mixing extension.
\end{rem}

We think that the following question has an affirmative answer.

\begin{ques} Assume $(X,T)$ is transitive and $(X,T)\in
\M_{pi}^\perp$. Is it true that $(X,T)$ is a weakly mixing
$E$-system?
\end{ques}

The difficulty to answer the question is that we do not know if each
subset of $\Z_+$ having lower Banach density 1 and containing $0$
contains a subset $A$ such that the orbit closure of $1_A$ is a
minimal PI system (there is such a set which does not contain any
subset $A$ such that the orbit closure of $1_A$ is a minimal HPI
system, since otherwise we have that scattering implies weak
mixing).

\subsection{Disjointness and weak disjointness for $\M$}

In \cite{HY} it was shown that a weakly mixing system with a dense
set of regular minimal points is disjoint from any minimal t.d.s..
Now we improve the result by showing that each weakly mixing system
with a dense set of distal points is disjoint from all minimal
systems. We give two proofs, where the first one is provided by W.
Huang and the second one relies on the structure theorem for minimal
systems. After that we will give another result on disjointness:
each $\F_{s}$-independent t.d.s. is disjoint from any minimal t.d.s.

First we will prove

\begin{thm}\label{huangidea}
Each weakly mixing system with a dense set of distal points is
disjoint from all minimal systems.
\end{thm}

To prove it we need the following Lemma \ref{proximalcell}
concerning proximal cell (see \cite{AK02, HSY}). Note that for a
t.d.s. $(X, T)$ and $x\in X$, $P[x]$ denotes the {\it proximal
cell}, i.e. $P[x]=\{y\in X: y\ \text{is proximal to}\ x\}=\{y\in X:
(x,y)\in P(X,T)\}$.

\begin{lem}\label{proximalcell}
Let $(X, T)$ be a weakly mixing t.d.s. Then for each
$x\in X$, $P[x]$ is a dense $G_\delta$ subset of $X$.
\end{lem}

\noindent {\em Proof of Theorem \ref{huangidea}}: Let $(X,T)$ be a
weakly mixing system with a dense set of distal points and
$\{x_s\}_{s=1}^\infty$ be a dense set of distal points. By Lemma
\ref{proximalcell} there is $x\in \bigcap_{s=1}^\infty P[x_s]$. Let
$(Y,S)$ be a minimal t.d.s. and $J\subset X\times Y$ be a joining.
Then there is $y\in Y$ such that $(x,y)\in J$. For each $x_s$,
$(x,x_s)$ is proximal, thus for each $\ep>0$,
$$\{n\in\Z_+: d(T^nx, T^nx_s)<\ep/2\}$$
is thick. Since $x_s$ is a distal point, $(x_s,y)$ is minimal and
hence $$\{n\in\Z_+: d(T^nx_s,x_s)<\ep/2, d(T^ny, y)<\ep\}$$ is
syndetic. Thus, for a given $\ep>0$ there exists $n\in \N$ such
that
$$d(T^nx, T^nx_s)<\ep/2,\ d(T^nx_s,x_s)<\ep/2,\ \text{and}\ d(T^ny,
y)<\ep.$$ That is, $d(T^nx, x_s)<\ep$ and $d(T^ny,y)<\ep$. This
implies that $(x_s,y)\in W=:\overline{orb((x,y),T\times S)}$, and
thus $X\times \{y\}\subset W\subset J$. It follows that $J=X\times
Y$ since $(Y,S)$ is minimal. Hence $(X,T)$ is disjoint from $(Y,S)$.
\hfill $\square$

\medskip

Now we give the second proof. Since by Theorem \ref{pid} each weakly
mixing system with a dense set of distal points is disjoint from any
PI minimal system, by the structure theorem for minimal systems
(Theorem \ref{structure}) we need to deal with weakly mixing RIC
extensions.

\begin{lem}\label{wm-extension}
Let $\pi: Y'\rightarrow Y$ be a weakly mixing RIC extension of
minimal systems. Then there is a dense $G_\delta$ subset $Y_0\subset
Y$ such that for each $y\in Y_0$ and each $x\in \pi^{-1}(y)$,
$P_{Y'}[x]$ is dense in the fibre $\pi^{-1}(y)$.
\end{lem}
\begin{proof} See Appendix.\end{proof}

The following proposition concerns the ``lifting'' of disjointness
by weakly mixing RIC extensions. Note that each t.d.s. $(X,T)$ has a
natural extension $(X',T')$ such that $T'$ is a homeomorphism. We
may assume that all t.d.s. are invertible when considering
disjointness, see \cite[Proposition 1.1]{HY}.

\begin{prop}\label{disjoint-wm}
Let $(X,T)$ be a t.d.s. with a dense set of distal points and let
$\pi: (Y', S')\rightarrow (Y,S)$ be a weakly mixing RIC extension of
minimal systems. Then
\begin{equation*}
    X\perp Y' \quad \text{ if and only if }\quad X\perp Y.
\end{equation*}
\end{prop}

\begin{proof}
It suffices to show if $X\perp Y$ then $X\perp Y'$. Let $J\subset
X\times Y'$ be a joining of $X$ and $Y'$. Let $x$ be a distal point
of $X$ and $y\in Y_0$, where $Y_0$ is defined in Lemma
\ref{wm-extension}. We remark that $Y_0$ is residual in $Y$. Since
$X\perp Y$, $id\times \pi (J)=X\times Y$. Thus there is some $y_0\in
Y'$ such that $(x,y_0)\in J$ and $\pi(y_0)=y$. Let $y'\in
P_{Y'}[y_0]\cap \pi^{-1}(y)$. Then $(x,y_0), (x, y')$ are proximal.
Since $x$ is distal, $(x,y')$ is minimal. And this implies that
$$(x,y')\in \overline{orb((x,y_0), T\times S')}\subset J.$$

By Lemma \ref{wm-extension}, such $y'$ is dense in $\pi^{-1}(y)$.
Thus $\{x\}\times \pi^{-1}(y)\subset J$. Since $y\in Y_0$ is
arbitrary  and $Y_0$ is residual, we have $\{x\}\times Y'\subset J$.
Finally, by the density of distal points in $X$, we have $J=X\times
Y'$.
\end{proof}

Now Theorem \ref{huangidea} follows from the structure theorem
(Theorem \ref{structure}), Theorem \ref{pid} and Proposition
\ref{disjoint-wm}.








\medskip

To prove another disjointness result we need some notions and
results from \cite{HLY}.

\begin{de}\label{indede}
Let $(X,T)$ be a t.d.s.. For a tuple $\bfA=(A_1,\ldots,A_k)$ of
subsets of $X$, we say that a subset $F\subseteq \Z_+$ is an {\it
independence set} for $\bfA$ if for any nonempty finite subset
$J\subseteq F$, we have $$\bigcap_{j\in
J}T^{-j}A_{s(j)}\not=\emptyset$$ for any $s\in \{1,\dots,k\}^J$.
Denote the collection of all independence sets for $\bfA$ by
$\Ind(A_1,\ldots, A_k)$ or $\Ind \bfA$.
\end{de}

\begin{de} \label{tuple:def}
Let $\F$ be a family, $k\in\N$ and $(X,T)$ be a t.d.s.. A tuple
$(x_1,\ldots,x_k)\in X^k$ is called an {\em $\F$-independent tuple}
if for any neighborhoods $U_1,\ldots,U_k$ of $x_1,\ldots,x_k$
respectively, one has $\Ind(U_1,\ldots,U_k)\cap\F\not=\emptyset$.

A t.d.s. $(X,T)$ is said to be {\em $\F$-independent of order $k$},
if for each tuple of nonempty open subsets $U_1,\ldots,U_k$ of $X$,
$\Ind(U_1,\ldots,U_k)\cap\F\not=\emptyset$, and $(X,T)$ is said to
be {\em $\F$-independent}, if it is {\em $\F$-independent} of order
$k$ for each $k\in\N$.
\end{de}

It is proved in \cite{HLY} that an $\F_{s}$-independent t.d.s. is
weakly mixing, has positive entropy and has a dense set of minimal
points. Moreover, the following lemma is proved.

\begin{lem} \label{no minimal syndetic tds:lemma}
For every minimal subshift $X\subseteq \Sigma_2$,
$\Ind([0]_X,[1]_X)$ does not contain any syndetic set.
\end{lem}

An easy consequence of Lemma \ref{no minimal syndetic tds:lemma} is
that there is no non-trivial minimal t.d.s. which is
$\F_{s}$-independent. Now we are ready to show

\begin{thm} \label{independence}
Each $\F_{s}$-independent of order 2 t.d.s. is disjoint from all
minimal systems.
\end{thm}

\begin{proof} Since it is an open question if an $\F_{s}$-independent
pair can be lifted by extensions, the proof of \cite{B2} can not be
applied here directly. We will use ideas of the proof in \cite{B2}
and Lemma \ref{no minimal syndetic tds:lemma}.

Let $(X,T)$ be an $\F_{s}$-independent t.d.s. and $(Y,S)$ be
minimal. Assume the contrary that $X\not\perp Y$. Then there is a
joining $J\not=X\times Y$. We may assume that $J$ is minimal, i.e.
if $J'$ is a joining and $J'\subset J$ then $J'=J$. For $x\in X$ let
$J[x]=\{y\in Y: (x,y)\in J\}$. We claim that there exists $x\in X$
such that $J[x]\cap J[Tx]=\emptyset$.

Now suppose that $J[x]\cap J[Tx]\not=\emptyset$ for all $x\in X$.
Let
$$J'=\bigcup_{x\in X} \{x\}\times \big( J[x]\cap J[Tx]\big).$$ It is
easy to check that $J'\subset J$ is a joining, and hence by
minimality $J'=J$. This implies that $J=X\times Y$, a contradiction.
So there exists $x\in X$ such that $J[x]\cap J[Tx]=\emptyset$.

There exist disjoint closed neighborhoods $W_0$ and $W_1$ of $x$ and
$Tx$ such that $J[W_0]\cap J[W_1]=\emptyset$, since $J$ is closed
and $J[x]\cap J[Tx]=\emptyset$. So there is an syndetic subset $S\in
\Ind(W_0,W_1)$. Let $\pi_X:J\ra X$ and $\pi_Y:J\ra Y$ be the
projections. It is clear that $S\in \Ind(\pi_X^{-1}(W_0),
\pi_X^{-1}(W_1))$ and $S\in \Ind(\pi_Y\pi_X^{-1}(W_0),
\pi_Y\pi_X^{-1}(W_1)).$ Since $J[W_0]\cap J[W_1]=\emptyset$ we know
that $\pi_Y\pi_X^{-1}(W_0)\cap \pi_Y\pi_X^{-1}(W_1)=\emptyset$. Let
$V_0$ and $V_1$ be disjoint closed neighborhoods of
$\pi_Y\pi_X^{-1}(W_0)$ and $\pi_Y\pi_X^{-1}(W_1)$ respectively. It
is clear that $S\in \Ind(V_0,V_1)$.

It is well known that we can find a minimal t.d.s. $(X_1, T_1)$ and
a factor map $\pi: (X_1, T_1)\rightarrow (Y,S)$ such that $X_1$ is a
closed subset of a Cantor set. 
It is easy to see that $\Ind(V_0, V_1)=\Ind(\pi^{-1}(V_0),
\pi^{-1}(V_1))$. Write $X_1$ as the disjoint union of clopen subsets
$U_0$ and $U_1$ such that $U_j\supseteq \pi^{-1}(V_j)$ for $j=0, 1$.
Then $\Ind(V_0, V_1)\subseteq \Ind(U_0, U_1)$.

Define a coding $\phi:X_1\ra \Sigma_2$ such that for each $x\in
X_1$, $\phi(x)=(x_0,x_1,\ldots)$, where $x_i=j$ if $T_1^i(x)\in U_j$
for all $i\in\Z_+$. Then $Z=\phi(X_1)$ is a minimal subshift
contained in $\Sigma_2$ and $\phi:X_1\rightarrow Z$ is a factor map.
It is easy to verify that $\Ind(U_0, U_1)\subseteq
\Ind([0]_Z,[1]_Z)$.

By Lemma~\ref{no minimal syndetic tds:lemma} we know that
$\Ind([0]_Z, [1]_Z)$ does not contain any syndetic set. This
contradicts the fact that $S\in\Ind([0]_Z, [1]_Z)$. So $X$ and $Y$
are disjoint.
\end{proof}

We remark that the assumption of $\F_s$-independence can not be
weaken significantly, since there exists an $\F_{pd}$-independent
t.d.s. with only one minimal point \cite{HLY}. So combining the
result in \cite{HY} we have
\begin{prop} The following statements hold:
\begin{enumerate}

\item Each weakly mixing system with a dense set of distal
points is disjoint from all minimal systems; and each
$\F_{s}$-independent t.d.s. is disjoint from all minimal systems.

\item If $(X,T)$ is transitive and is disjoint from any
minimal t.d.s. then $(X,T)$ is weakly mixing and has a dense set of
minimal points.

\end{enumerate}

\end{prop}
Recall that a t.d.s. is scattering if it is weakly disjoint from
$\M$. In \cite{BHM} the following proposition was proved. Recall
that a cover is non-trivial if each element of the cover is not
dense in $X$, and for a cover $\U$, $N(\U)=\min\{|\mathcal{V}|:
\mathcal{V}\ \text{is a subcover of}\ \U\}$.

\begin{prop} A t.d.s. is scattering if and only if for any
non-trivial open cover $\U$,
$N(\bigvee_{i=0}^{n-1}T^{-i}\U)\ra\infty$.
\end{prop}

\subsection{Disjointness and weak disjointness for $\M_{eq}$}  Recall that
a t.d.s. is {\it weakly scattering} if it is weakly disjoint from
$\M_{eq}$. The following proposition is known, see for example
\cite{AG}.

\begin{prop} A transitive t.d.s. is disjoint from $\M_{eq}$ if and only if it is
weakly scattering.
\end{prop}

Let $(X,T)$ and $(Y,S)$ be two transitive t.d.s.. If there exists a
continuous map $\phi:Tran_T(X)\rightarrow Tran_S(Y)$ with
$\phi(Tx)=S\phi(x)$ for $x\in Tran_T(X)$, then we say $\phi$ is a
{\it generic homomorphism} from $(X,T)$ to $(Y,S)$, $(Y,S)$ is a
{\it generic factor} $(X,T)$ and $(X,T)$ is a {\it generic
extension} of $(Y,S)$. It is not hard to see that if $(X,T)$ is
minimal and $\phi:(X,T)\rightarrow (Y,S)$ is a generic homomorphism
then $\phi$ is a factor map.

\medskip
In \cite{HY3} the authors considered weakly scattering t.d.s.. The
following proposition was a result in \cite{HY3} combing with a
simple observation.

\begin{prop} The following hold.

\begin{enumerate}
\item A transitive t.d.s. is weakly scattering if and only if it has
no non-trivial generic equicontinuous factors.

\item A minimal t.d.s. is disjoint from $\M_{eq}$ if and only if it
is weakly mixing.
\end{enumerate}
\end{prop}
\begin{proof} (1) was proved in \cite{HY3}. To show (2) note that if
a minimal t.d.s. is disjoint from $\M_{eq}$ then the maximal
equicontinuous factor of $(X,T)$ is trivial, which implies that
$(X,T)$ is weakly mixing. There are several ways to show a weakly
mixing t.d.s. is disjoint from $\M_{eq}$, say, for example
\cite{BHM, F}.
\end{proof}

It is clear scattering implies weak scattering. To end the
subsection we recall an open question:
\begin{ques}
Does weak scattering implies scattering?
\end{ques}

\subsection{Disjointness for $\M_0$}
The following proposition was proved in \cite{HKY}.

\begin{prop} The following statements hold:

\begin{enumerate}
\item If a transitive $(X,T)\perp \M_0$ then it is weakly mixing and is an
$E$-system.

\item If $(X,T)$ is a transitive diagonal t.d.s. then $(X,T)\perp
\M_0$.

\item If $(X,T)$ is minimal and $(X,T)\perp\M_0$ then $(X,T)$ has
c.p.e.; and if $(X,T)$ is minimal and diagonal, then $(X,T)\perp
\M_0$.
\end{enumerate}
\end{prop}

\subsection{Disjointness for $\M_{wm}$}
Since $\M_d^\perp\cap \M=\M_{eq}^\perp \cap \M= \M_{wm}$ \cite{F},
it implies that $\M_{wm}^\perp\supset \M_d$. The following
proposition is known.

\begin{prop}\label{auslander}\cite{AuG77} A minimal t.d.s. is in
$\M_{wm}^\perp$ if and only if every non-trivial quasi-factor of $X$
has a non-trivial distal factor.
\end{prop}

Recall that a {\it quasi-factor} of $X$ is a minimal subset of
$(2^X,T)$, where $2^X$ is the collection of all non-empty closed
subsets of $X$ equipped with the Hausdorff metric.

\begin{de} A minimal point $x$ is a {\it quasi-distal point} if $(x,y)$
is minimal for every minimal $y$ who's orbit closure is weakly
mixing.
\end{de}

It is clear that a distal point is quasi-distal. Moreover, if
$(X,T)$ is minimal and $(X,T)\in \M_{wm}^\perp$ then each point of
$X$ is quasi-distal, since two minimal t.d.s. are disjoint then the
product is minimal. By \cite[Theorem 2.2]{G80}, there exists a
quasi-distal point which is not distal. Since any almost one-to-one
extension of a minimal equicontinuous systems is in $\M_{wm}^\perp$
(say the Denjoy minimal t.d.s.), it follows that there is a
quasi-distal point which is not weakly product recurrent. It is not
clear if a minimal weakly product recurrent point is quasi-distal.
We have the following theorem.

\begin{thm} Let $(X,T)$ be a  weakly mixing t.d.s. with dense
quasi-distal points, then $(X,T)\in \M_{wm}^\perp.$
\end{thm}
\begin{proof} Apply the proof of Theorem \ref{huangidea}.
\end{proof}

\medskip

It is well known that a t.d.s. $(X,T)$ is weakly mixing if and only
if $N(U,V)$ is thick \cite{F}. Weiss \cite{Wg} showed that if
$F\subset \Z_+$ is a thick set then there is a weakly mixing t.d.s.
$(X,T)\subset (\{0,1\}^{\Z_+},\sigma)$ such that $N([1],[1])\subset
F$. Huang and Ye \cite{HY1} showed that if $(X,T)$ is minimal then
$(X,T)$ is weakly mixing if and only if $N(U,V)$ has lower Banach
density 1. By Remark \ref{wmpoint} we have

\begin{lem}\label{wmm} Let $F\subset \Z_+$ be thickly syndetic. Then there are a
minimal weakly mixing $(X,T)\subset (\{0,1\}^{\Z_+},\sigma)$ and
$x\in X$ such that $N(x,[1])\subset F$.
\end{lem}

So in the transitive case we have the following corollary:

\begin{prop}\label{propwm}
If a transitive $(X,T)$ is disjoint from all minimal weakly mixing
t.d.s. then it is an $M$-system.
\end{prop}
Since a minimal equicontinuous systems is in $\M_{wm}^\perp$,
$(X,T)\in \M_{wm}^\perp$ does not imply that it is weak mixing.

The following question remains open:

\begin{ques}
Is it true that a transitive t.d.s. $(X,T)$ is disjoint from any
minimal t.d.s. if and only if $(X,T)$ is weakly mixing and has a
dense set of quasi distal points?
\end{ques}






\section{Tables}

\begin{table}[!h]
\caption{$\F$-product recurrence} \vspace*{1.5pt}
\begin{center}
\begin{tabular}{| p{2.5cm} | p{2.8cm} | p{2.5cm} |p{2.5cm}|p{2.5cm}|p{2.5cm}|}\hline
\multicolumn{1}{|c|} {$\F$}  & \multicolumn{1}{|c|} {$\F_{inf}$} &
\multicolumn{1}{|c|} {$\F_{pubd}$} & \multicolumn{1}{|c|}
{$\F_{ps}$} & \multicolumn{1}{|c|}
{$\F_{s}$}  \\
\hline
  Orbit closure of a $\F$-PR point  & minimal distal & minimal
& minimal & $M$-system  \\

\hline
  Orbit closure of a $\F$-$PR_0$ point  & minimal distal & minimal
& minimal & $E$-system  \\
 \hline
 \end{tabular}
\end{center}
\end{table}


\begin{table}[!h]
\caption{Disjointness and weak disjointness} \vspace*{1.5pt}
\begin{center}
\begin{tabular}{| p{2.3cm} | p{2cm} | p{2cm} |p{2cm}|p{2.2cm}|p{2.2cm}|}\hline
\multicolumn{1}{|c|} {$\mathfrak{F}$}  & \multicolumn{1}{|c|}
{$\M_{eq}$} & \multicolumn{1}{|c|} {$\M_{hpi}$} &
\multicolumn{1}{|c|} {$\M_{pi}$} & \multicolumn{1}{|c|}
{$\M_{wm}$} & \multicolumn{1}{|c|}{$\M$}  \\
\hline Properties of transitive systems in $\mathfrak{F}^\perp$  &
weak scattering & weak scattering & weak mixing + $E$-system?? &
$M$-system & weak mixing + $M$-system
\\ \hline Systems in $\mathfrak{F}^\perp$ & weak scattering &
scattering & weak mixing + $M$-system & weak mixing +
dense quasi-distal points & w.m. + dense distal points; $\F_s$-independent   \\
\hline Minimal systems in $\mathfrak{F}^\perp$ & weak mixing & weak
mixing & weak mixing & every non-trivial quasi-factor has a
non-trivial distal factor & trivial
\\ \hline
Systems weakly disjoint from $\mathfrak{F}$ & weak scattering & ?? &
?? & ??
       & scattering \\ \hline
 \end{tabular}
\end{center}
\end{table}

\section{More discussions}

\subsection{$(\F_1,\F_2)$-product recurrence}
In this subsection we discuss some generalizations of the notions
concerning product recurrence.

\begin{de}
Let $\F_1, \F_2$ be families and $(X,T)$ be a t.d.s. A point $x\in
X$ is called {\em ($\F_1$, $\F_2$)-product recurrent} if $(x,y)$
is $\F_2$-recurrent for any $\F_1$-recurrent point $y$ in some
t.d.s $(Y,S)$.
\end{de}

By the definition it is obvious that $\F$-product recurrence is
nothing but $(\F, \F_{inf})$-product recurrence. As we have seen in
this paper, for a family the property $\F$-PR may be very complex.
Hence it is more difficult to discuss the general case $(\F_1,
\F_2)$-PR. But if we assume $\F_1=\F_2$, then we can use the results
from \cite{AuF94, EEN}. To see this, let us recall some notions
first.

\medskip

Now we consider the $\rm Stone-\breve{C}ech$ compactification of the
semigroup $\Z_+$ with the discrete topology. The set of all
ultrafilters on $\Z_+$ is denoted by $\beta \Z_+$. Let $A \subset
\Z_+$ and define $\overline A = \{p \in \beta \Z_+:A \in p\}$. The
set $\{\overline A:A \subset \Z_+\}$ forms a basis for the open sets
(and also a basis for closed sets) of $\beta \Z_+$. Under this
topology, $\beta \Z_+$ is the {\it $Stone-\breve{C}ech$
compactification} of $\Z_+$. See \cite{A, AAG, EEN} etc. for
details.

For $F \subset \Z_+$ the {\it hull of F} is $h(F)=\overline F=\{p\in
\beta \Z_+: F \in p \}$. For a family $\F$, the {\it hull of $\F$}
is defined by
\begin{equation*}
h(\F)=\bigcap_{F\in \F}h(F) = \bigcap_{F\in \F}\overline F=\{p\in \b
\Z_+ :\F \subseteq p\} \subseteq \b \Z_+ .
\end{equation*}

Let $X$ be a compact metric space and $S$ a semigroup. Let $\Phi : S
\times X \rightarrow X$ be an action, i.e. for any $p,q \in S$,
$\Phi ^p \circ \Phi^ q=\Phi^{pq}$. For $(p,x)\in S \times X $,
denote
\begin{equation*}
px = \Phi (p,x) =\Phi ^p (x)=\Phi _x(p).
\end{equation*}

$\Phi ^{\#} : S \rightarrow X^X $ is defined by $p \mapsto \Phi^p$.
Hence $px=\Phi ^{\#}(p)(x)$. An {\it Ellis semigroup $S$} is a
compact Hausdorff semigroup such that the right translation map
$R_p: S \longrightarrow S$, $q\longmapsto qp$ is continuous for
every $p \in S$. An {\it Ellis action} of an Ellis semigroup $S$ on
a space $X$ is a map $\Phi : S \times X \rightarrow X$ which is an
action such that the adjoint map $\Phi^{\#}$ is continuous, or
equivalently, $\Phi_x$ is continuous for each $x \in X$.

Now let $(X,T)$ be a t.d.s. Then $\Phi: \Z_+ \times X \rightarrow X,
(n,x)\mapsto T^nx$ is an action and it can be extended to an Ellis
action $\Phi: \b\Z_+ \times X\rightarrow X$. Hence we have a
continuous map $\Phi^{\#}:\b\Z_+ \rightarrow X^X$.

Define
\begin{equation*}
H(\F)=H(X, \F)=\Phi ^{\#}(h(\F))\subset X^X.
\end{equation*}
It is easy to see that for a family $\F$, $H(\F)\neq \emptyset$ if
and only if $\F$ has finite intersection property. Moreover, let
$(X,T)$ be a t.d.s and $\F$ be a filter. Then $H(\F)=\bigcap
\limits_ {F\in \F}\overline {T^F} \subseteq X^X$, where $T^F=\bigcup
\{T^n| n \in F\}$.

Now we generalize the notion of $\omega$-limit set. Let $(X,T)$ be a
t.d.s and $\F$ be a family. Define
\begin{equation*}
\w_{\F}(x, T)=  \bigcap _{ F \in \F^* } \overline{T^F(x)}.
\end{equation*}
It is easy to show that if $\F$ is a filter, then $\w_{\F^*}(x,
T)=H(\F)x$. By the definition one has that a point $x \in X$ is
$\F$-recurrent if and only if $x \in \w_{\F}(x, T)$.

Now let $\F$ be a filterdual (i.e. its dual is a filter). Then a
point $x$ is $(\F,\F)$-product recurrent if and only if $(x,y)\in
\w_{\F}((x,y), T\times S)$ for any $y$ in some t.d.s. $(Y, S)$
satisfying $y\in \w_\F(y, S)$. That is, $x$ is $H(\F^*)$-product
recurrent defined in \cite{AuF94}. Thus we can use the results in
\cite{AuF94, EEN} to study $(\F,\F)$-PR points.

\subsection{Questions}
Here are some more questions. First we restate the following
question in \cite{HO}.
\begin{ques}
Is each weakly product minimal point distal?
\end{ques}

We conjecture that the above question has a negative answer. The
next question concerns disjointness.

\begin{ques}
Let $(X_1, T_1), (X_2, T_2)$ be t.d.s., and $(Y,S)$ be a minimal
t.d.s.. If $(X_1,T_1)\perp (Y,S)$ and $(X_2, T_2)\perp (Y,S)$, then
is it true that $$(X_1\times X_2, T_1\times T_2)\perp (Y,S)?$$ Or
for a class $\mathcal{T}$ of minimal systems, is finite product
closed in $\mathcal{T}^\perp$?
\end{ques}

\section{Appendix: Relative proximal cells}

In this appendix we study the relative proximal cell for an
independent interest, and on the way to do this, we give a proof of
Lemma \ref{wm-extension}. Here we will use some results from the
theory of minimal flows. This theory was mainly developed for group
actions and accordingly we assume that $T$ is a homeomorphism in
this appendix. Much of this work can be done for a general locally
compact group actions. We refer the reader to \cite{Au, G, V77, Vr}
for details.

\subsection{RIM extension}

Let $X$ be a compact metric space and let $M(X)$ be the collection
of regular Borel probability measures on $X$ provided with the weak
star topology. Then $M(X)$ is a compact metric space in which $X$ is
embedded by the mapping $x\mapsto \d_x$, where $\d_x$ is the dirac
measure at $x$. If $\phi: X\rightarrow Y$ is a continuous map
between compact metric spaces, then $\phi$ induces a continuous map
$\phi^*: M(X)\rightarrow M(Y)$ which extends $\phi$, where
$(\phi^*\mu)(A)=\mu(\phi^{-1}A)$ for all Borel sets $A\subseteq Y$.

Let $(X,T)$ be a t.d.s.. For each $\mu\in M(X)$, define
$(T\mu)(A)=\mu(T^{-1}A)$ for all Borel sets $A\subseteq X$. Then
$(M(X), T)$ is a t.d.s. too. And if $\pi: X\rightarrow Y$ is an
extension of t.d.s., then $\pi^*: M(X)\rightarrow M(Y)$ is also an
extension.

An extension $\pi: X\rightarrow Y$ of t.d.s. is said to have a {\em
relatively invariant measure} (RIM for short) if there exists a
continuous homomorphism $\lambda:Y\rightarrow M(X)$ of t.d.s. such
that $\pi^*\circ \lambda : Y\rightarrow M(Y)$ is just the (dirac)
embedding. In other words: $\pi$ is a RIM extension if and only if
for every $y\in Y$ there is a $\lambda_y\in M(X)$ with ${\rm supp}
\lambda_y\subseteq \pi^{-1}(y)$ and the map $y\mapsto \lambda_y:
Y\rightarrow M(X)$ is a homomorphism of t.d.s; this map $\lambda$ is
called a {\em section} for $\pi$. Note that $\pi: X\rightarrow
\{\star\}$ has a RIM if and only if $X$ has an invariant measure if
and only if $M(X)$ has a fixed point, where $\{\star\}$ stands for
the trivial system. An extension $\pi: X\rightarrow Y$ is called
{\em strongly proximal} if for every pair $\mu\in M(X)$ and $y\in Y$
with ${\rm supp} \mu\subseteq \pi^{-1}(y)$, a sequence $\{n_i\}$ can
be found such that $\lim T^{n_i}\mu$ is a point mass. It is easy to
see that each strongly proximal extension is proximal.

\medskip

Every extension of minimal systems can be lifted to a RIM extension
by strongly proximal modifications. To be precise, for every
extension $\pi:X\rightarrow Y$ of minimal systems there exists a
canonically defined commutative diagram of extensions (called the
{\em G-diagram} \cite{G75})
\begin{equation*}
\xymatrix {
X \ar[d]_{\pi}  &  X^\# \ar[l]_{\sigma}\ar[d]^{\pi^\#} \\
Y &  Y^\# \ar[l]^{\tau} }
\end{equation*}
with the following properties:
\begin{enumerate}
\item[(a)]
$\sigma$ and $\tau$ are strongly proximal;
\item[(b)]
$\pi^\#$ is a RIM extension;
\item[(c)] $X^\#$ is the unique minimal set in $R_{\pi \tau}=\{(x,y)\in X\times
Y^\#: \pi(x)=\tau (y)\}$ and $\sigma$ and $\pi^\#$ are the
restrictions to $X^\#$ of the projections of $X\times Y^\#$ onto $X$
and $Y^\#$ respectively.
\end{enumerate}

By a small modification we can assume that $\pi^\#$ is an open RIM
extension. We refer to \cite{G75, Vr} for the details of the
construction.

\subsection{Relative regionally proximal relation}

Let $\pi: (X,T)\rightarrow (Y,T)$ be t.d.s.. For $\ep>0$ let
$\Delta_\ep=\{(x,y)\in X\times X: d(x,y)< \ep\}$. Then the {\em
relative proximal relation} is
$$P_\pi=\bigcap_{n=1}^\infty \big( \bigcup_{i\in \Z}T^i\Delta_{1/n} \big)\cap R_\pi,$$
and the {\em relative regionally proximal relation} is
$$Q_\pi=\bigcap_{n=1}^\infty \overline{\big( \bigcup_{i\in \Z}T^i\Delta_{1/n} \big)\cap R_\pi}.$$
For $R\subseteq X\times X$ and $x\in X$, define $R[x]=\{x'\in X:
(x,x')\in R\}$. Define
$$U_\pi[x]=\bigcap_{n=1}^\infty \overline{\big( \bigcup_{i\in \Z}T^i\Delta_{1/n} \big)[x]\cap
\pi^{-1}(\pi(x))}.$$ In other words: $x'\in U_\pi[x]$ if and only if
there are sequences $\{x_i'\}$ in $\pi^{-1}(\pi(x))$ and $\{n_i\}$
in $\Z$ such that
$$x_i'\to x' \ \ \text{and}\ \ (T\times T)^{n_i}(x, x_i')\to (x,x).$$
It is clear that $P_\pi[x]\subseteq U_\pi[x]\subseteq Q_\pi[x]$.
Define
$$U_\pi=\{(x,x')\in R_\pi: x'\in U_\pi[x]\}.$$
The following is an open question \cite{V77}:

\begin{ques}
If $\pi: X\rightarrow Y$ is an open Bronstein extension (i.e.
$R_\pi$ has a dense set of minimal points), does $U_\pi[x]=Q_\pi[x]$
for all $x\in X$?
\end{ques}

One does not have an answer for this question, but one has the
following result.

\begin{prop}\cite[Theorem 1.5]{Mc78}\label{Mc}
Let $\pi: X\rightarrow Y$ be a RIM extension of minimal systems with
section $\lambda$, and let $y\in Y$ be such that ${\rm supp}
\lambda_y=\pi^{-1}(y)$. Then for all $x\in \pi^{-1}(y)$ we have
$U_\pi[x]=Q_\pi[x]$.
\end{prop}

The following lemma guarantees that there are lots of such $y$ in
Proposition \ref{Mc}.
\begin{lem}\cite[Lemma 3.3]{G75}\label{Glasner}
Let $\pi: X\rightarrow Y$ be a RIM extension of minimal systems with
section $\lambda$. Then there is a residual set $Y_0\subseteq Y$
such that $y\in Y_0$ implies ${\rm supp} \lambda_y=\pi^{-1}(y)$.
\end{lem}

\subsection{Relative proximal cell}

Let $(X, T)$ be a weakly mixing t.d.s.. Then for each $x\in X$, the
proximal cell $P[x]$ is a dense $G_\delta$ subset of $X$ \cite{AK02,
HSY} (under the minimality assumption this result was obtained in
\cite{F1}). Now we consider the relative case. Let $\pi:
X\rightarrow Y$ be an extension of t.d.s. and $x\in X$. Call
$P_\pi[x]$ the {\em relative proximal cell} of $x$.

\begin{ques}
If $\pi: X\rightarrow Y$ is an open weakly mixing extension of
minimal systems, does the relative proximal cell $P_\pi[x]$ is a
residual subset of $\pi^{-1}(\pi(x))$ for all $x\in X$?
\end{ques}

We do not have full answer for this question. But we have the
following results.

\begin{thm}\label{RIM-prox}
Let $\pi: X\rightarrow Y$ be a weakly mixing and RIM extension of
minimal systems. Then there is a residual set $Y_0\subseteq Y$ such
that for all $y\in Y_0$ and all $x\in \pi^{-1}(y)$ we have that
$P_\pi[x]$ is residual in $\pi^{-1}(y)$.
\end{thm}

\begin{proof}
By Proposition \ref{Mc} and Lemma \ref{Glasner}, there is a residual
set $Y_0\subseteq Y$ such that for all $y\in Y_0$ and all $x\in
\pi^{-1}(y)$ we have $U_\pi[x]=Q_\pi[x]$. Fix such $y$ and $x$. Now
$\pi$ is weakly mixing, hence $Q_\pi=R_\pi$. Thus
$U_\pi[x]=Q_\pi[x]=R_\pi[x]=\pi^{-1}(y)$. Since
$U_\pi[x]=\bigcap_{n=1}^\infty \overline{\big( \bigcup_{i\in
\Z}T^i\Delta_{1/n} \big)[x]\cap \pi^{-1}(y)}$, we have
$$\overline{\big( \bigcup_{i\in
\Z}T^i\Delta_{1/n} \big)[x]\cap \pi^{-1}(y)}=\pi^{-1}(y), \ \
\forall n\in \N.$$ Hence
$$P_\pi[x]=\bigcap_{n=1}^\infty \big( \bigcup_{i\in
\Z}T^i\Delta_{1/n} \big)[x]\bigcap \pi^{-1}(y)$$ is a residual
subset of $\pi^{-1}(y)$.
\end{proof}
Applying the above theorem we have
\begin{thm}\label{B-prox}
Let $\pi: X\rightarrow Y$ be an extension of minimal systems. If
$\pi$ is weakly mixing and Bronstein (i.e. $R_\pi$ has a dense set
of minimal points), then there is a residual set $Y_0\subseteq Y$
such that for all $y\in Y_0$ and all $x\in \pi^{-1}(y)$ we have
$P_\pi[x]$ is residual in $\pi^{-1}(y)$.
\end{thm}

\begin{proof}
To apply Theorem \ref{RIM-prox}, we consider the following
G-diagram:
\begin{equation*}
\xymatrix {
X \ar[d]_{\pi}  &  X^\# \ar[l]_{\sigma}\ar[d]^{\pi^\#} \\
Y &  Y^\# \ar[l]^{\tau} }
\end{equation*}
First we claim that $(\sigma\times \sigma)R_{\pi^\#}=R_{\pi}$. By
the commutativity of the diagram, we have $(\sigma\times
\sigma)R_{\pi^\#}\subseteq R_{\pi}$. Now we show the converse. Since
the minimal points of $P_\pi$ is dense in $R_\pi$ it is sufficient
to show that every minimal point of $R_\pi$ is an element of
$(\sigma\times \sigma)R_{\pi^\#}$. Let $(x_1, x_2)\in R_\pi$ be
minimal, then there is a minimal point $(x_1',x_2')\in X^\#\times
X^\#$ such that $(\sigma\times \sigma)(x_1',x_2')=(x_1,x_2)$. Hence
$(\pi^\#(x_1'), \pi^\#(x_2'))$ is a minimal point of $Y^\# \times
Y^\#$. But $\tau(\pi^\#(x_1'))=\tau(\pi^\#(x_2'))$ and $\tau$ is
proximal, and hence we have $\pi^\#(x_1'), \pi^\#(x_2')$ are
proximal. To conclude we have $\pi^\#(x_1')=\pi^\#(x_2')$, i.e.
$(x_1',x_2')\in R_{\pi^\#}$.

Since $\pi$ is weakly mixing, it can be shown that $\pi^\#$ is also
weakly mixing (for example, see \cite[VI(3.19)]{Vr}). Now $\pi^\#$
is weakly mixing and RIM, by Theorem \ref{RIM-prox}, there is a
residual set $Y_0^\#\subseteq Y^\#$ such that for all $y^\#\in
Y_0^\#$ and all $x^\# \in \pi^{-1}(y^\#)$ we have $P_{\pi^\#}[x^\#]$
is residual in $(\pi^\#)^{-1}(y^\#)$. Let $Y_0=\tau (Y_0^\#)$. Since
$Y^\#$ is minimal and hence $\tau$ is semi-open, $Y_0$ is also a
residual subset of $Y$. Let $y\in Y_0$ and $y^\#\in Y_0^\#$ with
$\tau(y^\#)=y$. Let $x\in \pi^{-1}(y)$. Since $(\sigma\times
\sigma)R_{\pi^\#}=R_{\pi}$, we have $\sigma
((\pi^\#)^{-1}(y^\#))=\pi^{-1}(y)$. There is some $x^\#\in
(\pi^\#)^{-1}(y^\#)$ such that $\sigma (x^\#)=x$. Since
$P_{\pi^\#}[x^\#]$ is dense in $(\pi^\#)^{-1}(y^\#)$, $P_\pi[x]$ is
dense in $\pi^{-1}(y)$. But $P_\pi[x]$ always is a $G_\d$ subset of
$\pi^{-1}(y)$, and hence it is residual in $\pi^{-1}(y)$. The proof
is completed.
\end{proof}

Lemma \ref{wm-extension} is now followed from Theorem \ref{B-prox},
since each RIC extension is Bronstein.


\end{document}